\documentclass[11pt,a4paper]{amsart}
\usepackage[utf8]{inputenc}
\usepackage[english]{babel}
\usepackage[T1]{fontenc}
\usepackage{url}
\usepackage{mathrsfs}
\usepackage{tikz}
\usepackage{xpatch}

\xpatchcmd{\proof}{\itshape}{\proofnamefont}{}{}
\newcommand{\proofnamefont}{\bfseries}

\usepackage{upgreek} 
\usepackage{amssymb} 
\usepackage{xfrac} 
\usepackage{accents} 
\usepackage{mathtools} 
\usepackage{fixltx2e} 
\usepackage{nameref} 
\usepackage{hyperref} 
\hypersetup{
linkcolor=blue,urlcolor=cyan, citecolor=blue, colorlinks=true}

\title[Behavior of Gordian graphs at infinity]{Behavior of Gordian graphs at infinity}
\author{Alexey Yu. Miller}
\thanks{The work is supported by Ministry of Science and Higher Education of the Russian Federation, agreement № 075–15–2019–1620}

\address{Alexey Yu. Miller\\
St.\,Petersburg Department of 
Steklov Institute of Mathematics\\
St.~Petersburg State University}
\email{miller.m2@mail.ru}

\newcommand \sm {\setminus}

\newcommand \be     {\begin{equation}}
\newcommand \ee     {\end{equation}}

\DeclareMathOperator{\G}{G}
\DeclareMathOperator{\un}{u}

\DeclareMathOperator{\bL}{\mathcal{L}}   
\DeclareMathOperator{\bK}{\mathcal{K}}   
\DeclareMathOperator{\dt}{d}
\DeclareMathOperator{\X}{X}

\DeclareMathOperator{\F}{F}
\DeclareMathOperator{\N}{N}
\DeclareMathOperator{\Mod}{mod}
\DeclareMathOperator{\Ll}{L}

\DeclareMathOperator{\U}{U}
\DeclareMathOperator{\e}{e}
\DeclareMathOperator{\Int}{Int}
\DeclareMathOperator{\Sf}{S}
\DeclareMathOperator{\B}{B}

\DeclareMathOperator{\UCC}{UCC}
\DeclareMathOperator{\Hm}{H}
\DeclareMathOperator{\C}{\mathcal{C}}

\DeclareMathOperator{\ff}{f}
\DeclareMathOperator{\Cl}{Cl}
\DeclareMathOperator{\BI}{BI}
\DeclareMathOperator{\FI}{FI}
\DeclareMathOperator{\FU}{FU}
\DeclareMathOperator{\BU}{BU}
\DeclareMathOperator{\Vrt}{Vert}
\DeclareMathOperator{\ICC}{ICC}
\DeclareMathOperator{\Cm}{C}
\DeclareMathOperator{\PF}{PF}
\DeclareMathOperator{\PB}{PB}

\newcommand{\ti}[1]{{{\bf#1}}} 
\newcommand{\ci}[1]{{{\accentset{\circ}#1}}} 
\newcommand{\ra}[2]{\langle\sfrac{#1}{#2}\rangle}

\newtheorem{thm}{Theorem}

\newtheorem{lemma}{Lemma}

\theoremstyle{definition}
\newtheorem{definition}{Definition}

\theoremstyle{remark}
\newtheorem*{remark}{Remark}

\newtheoremstyle{namedlem}{}{}{\itshape}{}{\bfseries}{.}{.5em}{\thmnote{#3}#1}
\theoremstyle{namedlem}
\newtheorem*{namedlem}{}

\usepackage{floatrow}
\graphicspath{{Pictures/}}
\DeclareGraphicsExtensions{.pdf,.png,.jpg,.eps}

\begin{document}
\maketitle
\begin{abstract}
The present paper refers to the knot theory and is devoted to the study of global properties of Gordian graphs of various local moves. In~2005, Gambaudo and Ghys raised the question of the behavior at infinity of the crossing change Gordian graph. They proposed studying its “ends”, that is, unbounded connected components of complements of bounded subsets. We provide a complete description of the behavior at infinity for local moves from three well-known infinite families, namely, rational moves, $\Cm(n)$-moves, and~$\Hm(n)$-moves (note that each of the first two families contains the crossing change). Also, in 2005, March\'e gave a different perspective on the behavior of Gordian graphs at infinity, proposing to consider complements of finite subsets. We describe the behavior at infinity in this sense for all local moves with the infinite neighborhood of the unknot in the corresponding Gordian graph.
\end{abstract} 
\maketitle

\section*{Introduction}

The present paper refers to the classical theory of knots and is devoted to the study of knot transformations. A knot and link transformation is any geometric procedure that transforms a given link into some new link, possibly the same. The most well-studied class of knot transformations is the class of local moves. A local move is such a knot transformation that is represented as a local removal of one tangle and replacing it with another tangle (see~\cite{Kaw96, NakC05} and a detailed definition below). Classical examples of local transformations are the crossing change (see Figure~\ref{Fig01} and~\cite{Wen37, Sch85, BCJTT2017}), band surgery (see~\cite{K10, K11, KM09}), $\Delta$-move (see~\cite{MN89, O90, Hor08}), $n$-move (see~\cite{Prz88, DP02, MWY19}), and many others. Local moves are the main object of our study. The main method for studying knot transformations is to study their Gordian graphs.

\begin{figure}[H]
\center{\includegraphics[width=0.5\textwidth]{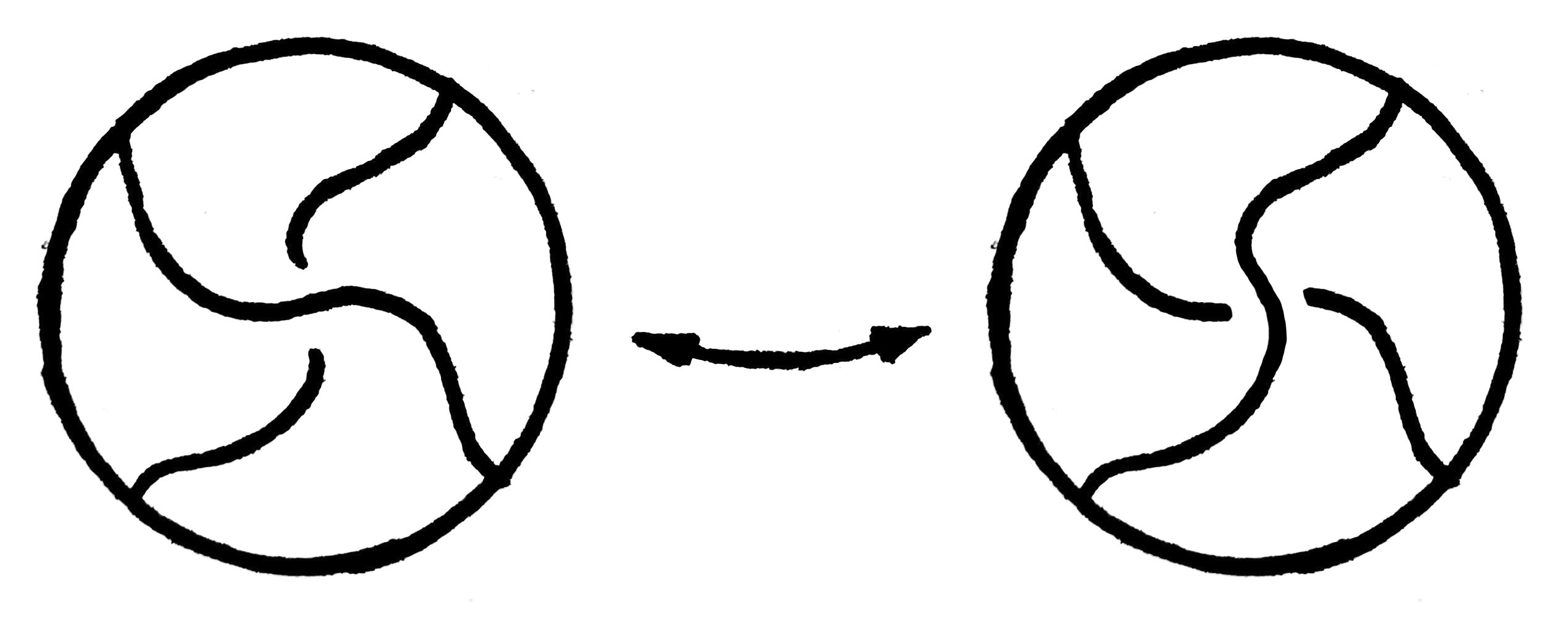}}
\caption{\ti{Crossing change}} 
\label{Fig01}
\end{figure}

\medskip
Let $\delta$ be a knot transformation, and let $M$ be a set of links. Then we denote by~$\G(\delta,M)$ such a graph whose vertex set is in one-to-one correspondence with~$M$, and two vertices are connected by an edge if and only if the corresponding links are obtained from each other (sic!) by a single application of $\delta$. The graph $\G(\delta,M)$ is called the Gordian graph for $\delta$ and $M$. We consider this graph with a natural metric. The distance between two vertices in this metric is the number of edges in any shortest path connecting these vertices if they are in the same connected component. Further, we neglect the difference between a vertex of the Gordian graph and its corresponding link, identifying these objects and perceiving them as a single object, but nevertheless using both of these words to refer to it.

\medskip
The classic questions about the structure of a Gordian graph are the following questions. What is the structure of a unit sphere centered at the unknot in the Gordian graph for some transformation (see~\cite{Kon79, Sto03, Gre09, KM86, Ba14, K11, KM09, Uch00, Mcc17, Tho89,  Lin96, OTY06, BGS13,  Miy98})? What is the structure of a unit sphere centered at an arbitrary knot in the Gordian graph for some transformation (see~\cite{BCJTT2017, BKMMF21, O06, ZYL17, ZY18})? What distance properties exist in the Gordian graph for some transformation (see~\cite{Kaw12, HO12, O90, Nak12, Ze15, Wen37, Ble84, Yam82, Yam86, Nak05, AHH12, Sto04})? 
The properties studied in these and many other (see~\cite{DP02, DIP07, MV20, JLS12, K10, K12, K16}) works we call local.
By local we mean those transformation properties that can be described without using Gordian graphs. In the case of local transformation properties, Gordian graphs are used rather as a visualization tool. However, not all transformation properties are local. Many transformation properties are difficult to describe without using the Gordian graph as a structure on the set of all links. For example, the hyperbolicity of the Gordian graph (see~\cite{JaLiMo22, IJ11, FMS21}), the existence of a metric filtration on the Gordian graph (see~\cite{BK20}), the existence of special complete subgraphs (see the theory of Gordian complexes,~\cite{HU02, NO09, NO06, NakC05}), the non-trivial geometric structure (see \cite{Hor08, HO14, HO13, Baa06, Y00}). We call such properties global properties.

\medskip
In 2005, Gambaudo and Ghys found another global property of $\X$-move, see~\cite[Theorem C]{GG05}. They proved that for every integer \mbox{$d\ge1$}, there is a map $$\mu\colon \mathbb{Z}^d\rightarrow\G(\X,\bK)$$ which
is a quasi-isometry onto its image. This result was obtained as part of the study of the $\X$-move Gordian graph global geometric structure and, in particular, the behavior of this graph at infinity. In addition, Gambaudo and Ghys presented some open questions (see~\cite[p.~547]{GG05}) related to this line of investigation. In particular, they propose to study the space of "ends" of $\G(\X,\bK)$, that is, consider unbounded connected components of the complements of large balls in $\G(\X,\bK)$.

\medskip
In this paper, we present a solution to this problem. In fact, we give a more general result. We describe "ends" for knot transformations from three well-known infinite families (two of which contain $\X$-move). Let us first explain how we formalized the notion of "ends". We introduce four indicators of graph behavior at infinity (in this paper all graphs are assumed to be finite or countable). Let~$G$ be a connected graph, then we define the elements $\BU(G)$, $\BI(G)$, $\FU(G)$, and~$\FI(G)$ of the set \mbox{$\{0\}\cup\mathbb{N}\cup\{\infty\}$} as follows
$$\BU(G)=\sup_{V\,{ \in }\,{\PB}(G)} \UCC\Big(G\sm V\Big),\,\,\,\,\BI(G)=\sup_{V\,{ \in }\,{\PB}(G)} \ICC\Big(G\sm V\Big),$$
$$\FU(G)=\sup_{V\,{ \in }\,{\PF}(G)} \UCC\Big(G\sm V\Big),\,\,\,\,\FI(G)=\sup_{V\,{ \in }\,{\PF}(G)} \ICC\Big(G\sm V\Big),$$
where $\UCC(G\sm V)$ is the cardinality~of the set of unbounded con\-nec\-ted components of $G\sm V$, \mbox{$\ICC(G\sm V)$} is the cardinality of the set of infinite connected components of $G\sm V$, ${\PF}(G)$ is the set of all finite subsets of $\Vrt(G)$, and ${\PB}(G)$ is the set of all bounded (as subsets of $G$) subsets of $\Vrt(G)$, and $G\sm V$ is the graph obtained by removing from $G$ the set of vertices $V$ and all edges adjacent to these vertices. We can extend our four definitions to the case of graphs with more than one connected component taking as a value the sum of the corresponding values over all connected components. Let $G$ be a graph, then we say that $\BU(G)$ is the number of \mbox{$\BU$-ends}, $\BI(G)$ is the number of $\BI$-ends, $\FU(G)$~is the number of $\FU$-ends, and $\FI(G)$ is the number of $\FI$-ends of $G$. It is easy to see that the "ends" from Gambaudo and Ghys's question are $\BU$-ends.

\medskip
We describe the $\BU$-ends of rational moves (see Definition~\ref{rationalmoves}, and \cite{DS00}) by the following theorem:
\begin{thm}[$\BU$-Ends of rational moves]
\label{thm1}
For any rational move $\delta$, the number of $\BU$-ends of each connected component of $\G(\delta,\bK)$ is equal to one.
\end{thm}

We describe the $\BU$-ends of $\Cm(n)$-moves (see Definition~\ref{cnmoves}, and \cite{OY08}) by the following theorem:
\begin{thm}[$\BU$-Ends of $\Cm(n)$-moves]
\label{thm2}
For any $n\in \mathbb{N}$, the number of $\BU$-ends of each connected component of $\G(\Cm(n),\bK)$ is equal to one.
\end{thm}

We describe the $\BU$-ends of $\Hm(n)$-moves (see Definition~\ref{hnmoves}, and \cite{HNT90}) by the following theorem:
\begin{thm}[$\BU$-Ends of $\Hm(n)$-moves]
\label{thm3}
For any $n\in \mathbb{N}$, the number of $\BU$-ends of each connected component of $\G(\Hm(n),\bK)$ is equal to one.
\end{thm}

The proof of these theorems is based on \nameref{anothertanglebr}, which is formulated and proved in Section~\ref{sec:branched}, \nameref{shift}, and \nameref{basiclem}, which are formulated and proved in Section~\ref{sec:main}.

\medskip

In 2005, Marché answered some of the questions posed by Gambaudo and Ghys. In particular, he proved that if $V$ is a finite set of knots, then $\G(\X,\bK\sm V)$ is connected, see~\cite{Ma05}. We generalize this result to a large class of knot transformations by the following theorem, where by $\Sf_1^{\delta}(\U)$ we denote the set of knots adjacent to the unknot in $\G(\delta, \bK)$ (see Definition~\ref{ballandsphere}):

\begin{thm}[$\FI$-Ends and $\FU$-Ends]
\label{thm4}
Let $\delta$ be a local move such that $\Sf_1^{\delta}(\U)$ is an infinite set then the number of $\FI$-ends and the number of $\FU$-ends of each connected component of $\G(\delta, \bK)$ are equal to one.
\end{thm}

In addition, we have a conjecture that the number of $\BU$-ends of $\G(\delta,\bK)$ is equal to the number of its connected components for any almost trivial move~$\delta$ (see Definition~\ref{almosttrivialmoves}). We could also ask for the number of $\BI$-ends of a knot transformation.

\section*{Structure of the paper}
In Section~\ref{sec:preliminaries}, we recall some basic definitions of knot theory.

In Section~\ref{sec:rational},  we recall some basic concepts of the theory of rational tangles.

In Section~\ref{sec:gordian}, we give the formal definition of the Gordian graph and introduce a notion of "ends" of a graph.

In Section~\ref{sec:local}, we give the formal definition of a local move, give some examples of well-known local moves, and introduce a new family of local moves called almost trivial moves.

In Section~\ref{sec:branched}, we recall the concept of a branched covering and give some related results that we need. Also, we prove \nameref{anothertanglebr}.

In Section~\ref{sec:alexander}  we define the Alexander polynomial and the Conway polynomial and give some related results that we need. 

In Section~\ref{sec:main}, we prove \nameref{shift}, \nameref{basiclem}, Theorem~\ref{thm1}, Theorem~\ref{thm2}, Theorem~\ref{thm3}, Theorem~\ref{thm4}.

\section*{Acknowledgments}
The author is deeply indebted to his advisor, Dr.~Andrei Malyutin, for his guidance, patience, insight, and support. The author grateful to \mbox{Arshak} Aivazian, Ilya Alekseev, Vasilii Ionin and other participants of the Low-dimensional topology student seminar of the Leonhard Euler International Mathematical Institute in Saint Petersburg for helpful discussions.

\section{Preliminaries}\label{sec:preliminaries}
In this section, we recall some of the basic concepts, objects, and constructions of knot theory, that we need.  By a link we mean a piecewise-smooth embedding of a disjoint union of a finite number of circles into an oriented three-dimensional sphere $S^3$. We also use the term link to refer to the image of this embedding considered up to ambient isotopy. In this paper, all links are assumed to be tame and unoriented unless said otherwise. By a knot we mean a one-component link. We use the notation $\U$ for the unknot and denote by $\bL$ the set of all links, by~$\bL^\circ$ the set of all oriented links, and by $\bK$ the set of all knots (here, by the set of all links we mean, of course, the countable set of all isotopy classes of links). Let $K$ be a knot in $S^3$. An orientable surface $M$ in $S^3$ is called a \emph{Seifert surface} for $K$ if $\partial M=K$. It is known that any knot has an associated Seifert surface, see~\cite{Ro03}.

\begin{definition}[connected sum]
\label{consum}
Let $K$ and $Q$ be knots in $S^3$. We say that a knot~$W$ in $S^3$ is a \emph{connected sum} of $K$ and $Q$ if there are knots $K'$, $Q'$, and~$W'$ in~$S^3$ and a $3$-ball $B$ in~$S^3$ such that $K$, $Q$, and $W$ are ambient isotopic to \mbox{$K'$, $Q'$,} and~$W'$, respectively, $K'$ lies in $B$, $Q'$ lies in $S^3\sm \Int(B),$ $K'\cap Q'$ is a simple arc~$\gamma$ lying in $\partial B$, and $W'=(K'\cup Q')\sm\Int(\gamma).$ 
\end{definition}

\begin{remark}
Note that in the general case there can be two distinct knots each of which is a connected sum of $K$ and $Q$. This problem of ambiguity is solved by choosing an orientation on $K$ and $Q$, but due to the specifics of our further reasoning and the desire to work with unoriented knots we are also satisfied with this not a very clear definition.
\end{remark}

\begin{definition}[tangle] An \emph{$n$-tangle} is a pair $B_A=(B,A)$, where $B$ is a three-di\-men\-sional ball and $A$~is a collection of $n$ disjoint arcs embedded in $B$ such that the sphere $\partial B$ intersects with each arc at the endpoints of this arc and only at them. We call $B$ the \emph{base ball} of $B_A$. The \emph{strings} of $B_A$ are connected components of~$A$.    
\end{definition}

\begin{remark}
Two $n$-tangles $B_A$ and $B_C$ are said to be \emph{isotopic} if the set of endpoints~$\partial A$ coincides with $\partial C$, and if there is an ambient isotopy of $(B, A)$ to~$(B, C)$ that is the identity on the boundary $(\partial B,\partial A)=(\partial B, \partial C)$. Note that we also use the same term 
 tangle to denote an equivalence class with respect to isotopy. In this connection, below we use the notation $T=S$ for a pair of isotopic tangles $T$ and $S$, that means both the isotopy of the representatives and the equality of the corresponding isotopy classes. 
\end{remark}

\section{Rational tangles}\label{sec:rational}
In this section, we focus on a special class of $2$-tangles called rational tangles. We define some basic concepts of the rational tangles theory and give some necessary results. A more detailed survey can be found in~\cite{KL02}, \cite[p.~21]{Kaw96}, \cite{Con70}, \cite[p.~189]{BZ85}, \cite[p.~189]{Cro04}, \cite[p.~171]{Mur96}. It is worth noting that there is no universally accepted notation in this theory. To avoid confusion, we note that we use the notation of~\cite{KL02} and \cite[p.~21]{Kaw96}.

Let us first give a formal definition of a rational tangle. A \mbox{$2$-tangle} $B_A$ is called \emph{rational} if there is an orientation-preserving homeomorphism of pairs
$$ h\colon (B,A)\rightarrow(D^2\times I,\{x,y\}\times I) $$ 
where $I$ is a unit interval, $D^2$ is a unit disk, $x$ and $y$ are two distinct points on~$D^2$. Denote by $[0]$ and $[\infty]$ the two simplest rational tangles whose diagrams are shown in Figure~\ref{Fig03}. These two tangles are called \emph{trivial}. It is easy to see that a rational tangle is just a tangle that can be obtained by applying a finite number of consecutive twists of neighbouring endpoints to either $[0]$ or $[\infty]$. 

\begin{figure}[H]
\center{\includegraphics[width=0.6\textwidth]{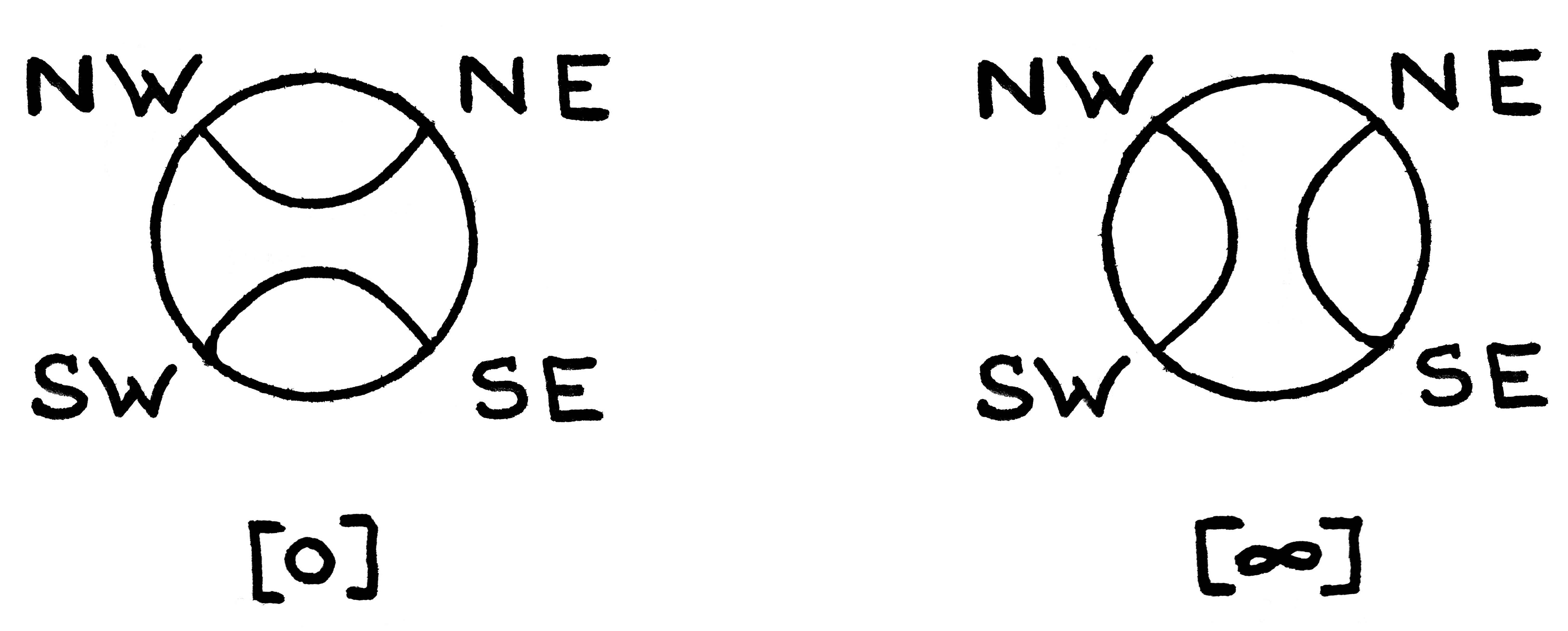}}
\caption{\ti{Trivial tangles}} 
\label{Fig03}
\end{figure}

For a more precise description of rational tangles, we need the following auxiliary notions. There are two operations on $2$-tangles called the \emph{addition} and the~\emph{star-pro\-duct}. The addition of $2$-tangles~$T$ and~$S$ is performed by attaching the two right endpoints of~$T$ to the two left endpoints of~$S$ as shown in Figure~\ref{Fig04}. The result of addition is denoted by~$T+S$. The star-product of $2$-tangles~$T$ and~$S$ is performed by attaching the two lower endpoints of~$T$ to the two upper endpoints of~$S$ as shown in Figure~\ref{Fig04}. The result of star-product is denoted by~$T*S$. 

\begin{figure}[H]
\center{\includegraphics[width=0.5\textwidth]{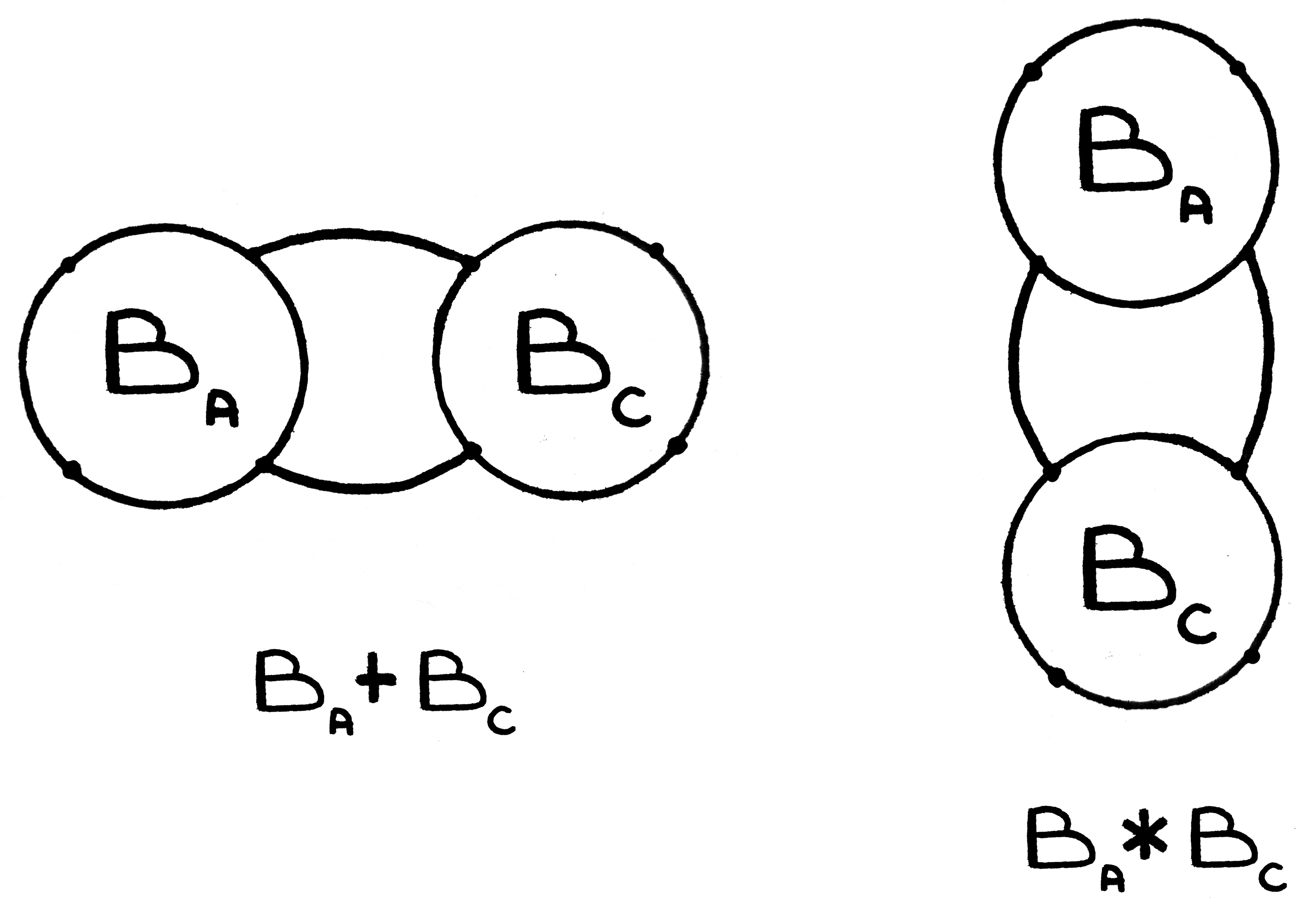}}
\caption{\ti{Operations}} 
\label{Fig04}
\end{figure}

We denote by~$[1]$ and~$[-1]$ the tangles shown in Figure~\ref{Fig05} to the right and left of $[0]$, respectively. A $2$-tangle is called \emph{integer} (resp.~\emph{vertical}) if it can be obtained from $[0]$ (resp.~from $[\infty]$) by
a finite number of consecutive additions of (resp.~multiplications by) $[-1]$ or $[1]$. Let $n$ be a non-negative integer. The sum of~$n$ copies of $[1]$ (resp.~of $[-1]$) is denoted by $[n]$ (resp.~by $[-n]$). The \mbox{star-product} of $n$ copies of $[1]$ (resp.~of $[-1]$) is denoted by $\frac{1}{[n]}$ (resp.~by $\frac{1}{[-n]}$). Note that for any integer tangle $T$ there exists $m\in\mathbb{Z}$ such that $T=[m]$ and for any vertical tangle~$S$ there exists $k\in\mathbb{Z}$ such that $S=\frac{1}{[k]}$. We also note that by defini\-tion,~$\frac{1}{[0]}=[\infty]$. We can now give a description of the canonical form for a rational tangle as follows. 

\begin{figure}[H]
\center{\includegraphics[width=0.8\textwidth]{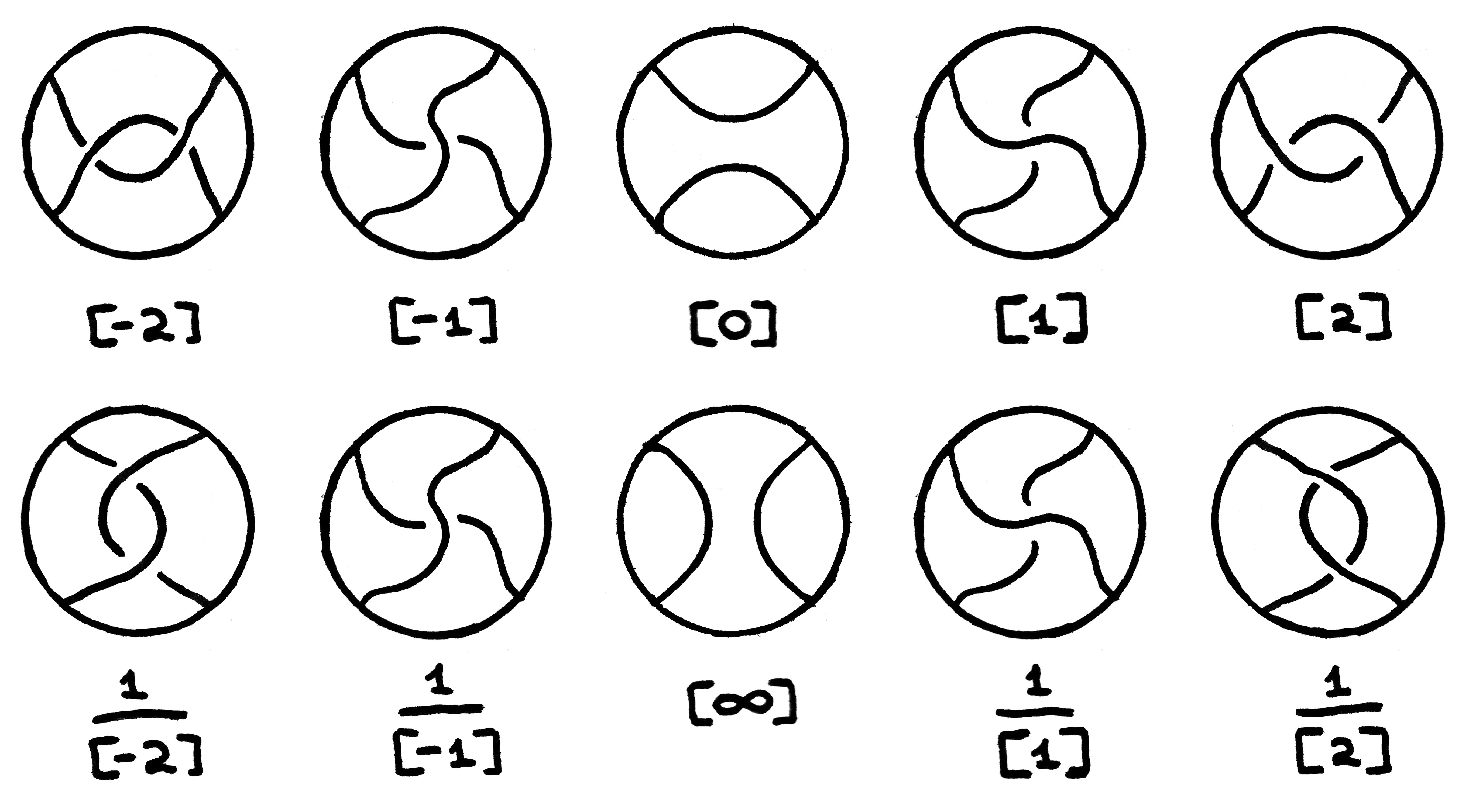}}
\caption{\ti{Integer and vertical tangles}} 
\label{Fig05}
\end{figure}

\begin{lemma}[Algebraic canonical form (\cite{KL02})]
\label{lem:ACF}
Let $T$ be a rational tangle. Then there is an odd number $n$ and there are 
 $a_1\in\mathbb{Z}$ and $a_2,\,\dots,\,a_{n}\in \mathbb{Z}\setminus\{0\}$ such that the $a_i$’s are all  positive  or all negative and
$$\big(\big(\big([a_n]*\frac{1}{[a_{n-1}]}\big)+[a_{n-2}]\big)*\dots*\frac{1}{[a_2]}\big)+[a_1]=T.$$
\end{lemma}

We also define three more operations on 2-tangles and describe their properties. The \emph{mirror image} of a $2$-tangle $T$ is denoted by $-T$ and it is obtained by switching all the crossings in $T$. The \emph{rotate} of $T$ is denoted by $T^R$ and it is a tangle obtained by counter-clockwise rotation of $T$ by $90^\circ$. The \emph{inverse} of $T$ is denoted by $\frac{1}{T}$ and it is defined to be $-T^R$. Note that all these operations preserve the class of rati\-onal tangles. Figure~\ref{Fig07} shows examples of their application. It is easy to see that we have 
$$T=\frac{1}{\frac{1}{T}},\,\,\,\,\,\,\,\,\,\,\,T^R=\frac{1}{-T}=-\frac{1}{T}.$$ 

In addition we note that all of the above notation (for addition, multiplication, inversion, and mirror image) is consistent with the notation for integer and vertical tangles, namely, if $n,m\in\mathbb{Z}$, then
$$[n]^i =\frac{1}{[n]},\,\,\, \Bigg(\frac{1}{[n]}\Bigg)^i=[n],\,\,\,-[n]=[-n],\,\,\,-\frac{1}{[n]}=\frac{1}{[-n]},$$
$$[n]+[m]=[n+m],\,\,\,\,\frac{1}{[n]}*\frac{1}{[m]}=\frac{1}{[n+m]}$$
and so on. And we have the following lemma.

\begin{lemma}[Operation Properties (\cite{KL02})]
\label{lem:OP}
Let $T$ be a rational tangle, and let $n\in\mathbb{Z}$, then we have
$$[n]+T=T+[n],\,\,\,\,\,\,\,\,\,T*[n]=\frac{1}{[n]+\frac{1}{T}}.$$
\end{lemma}

\begin{figure}[H]
\center{\includegraphics[width=0.6\textwidth]{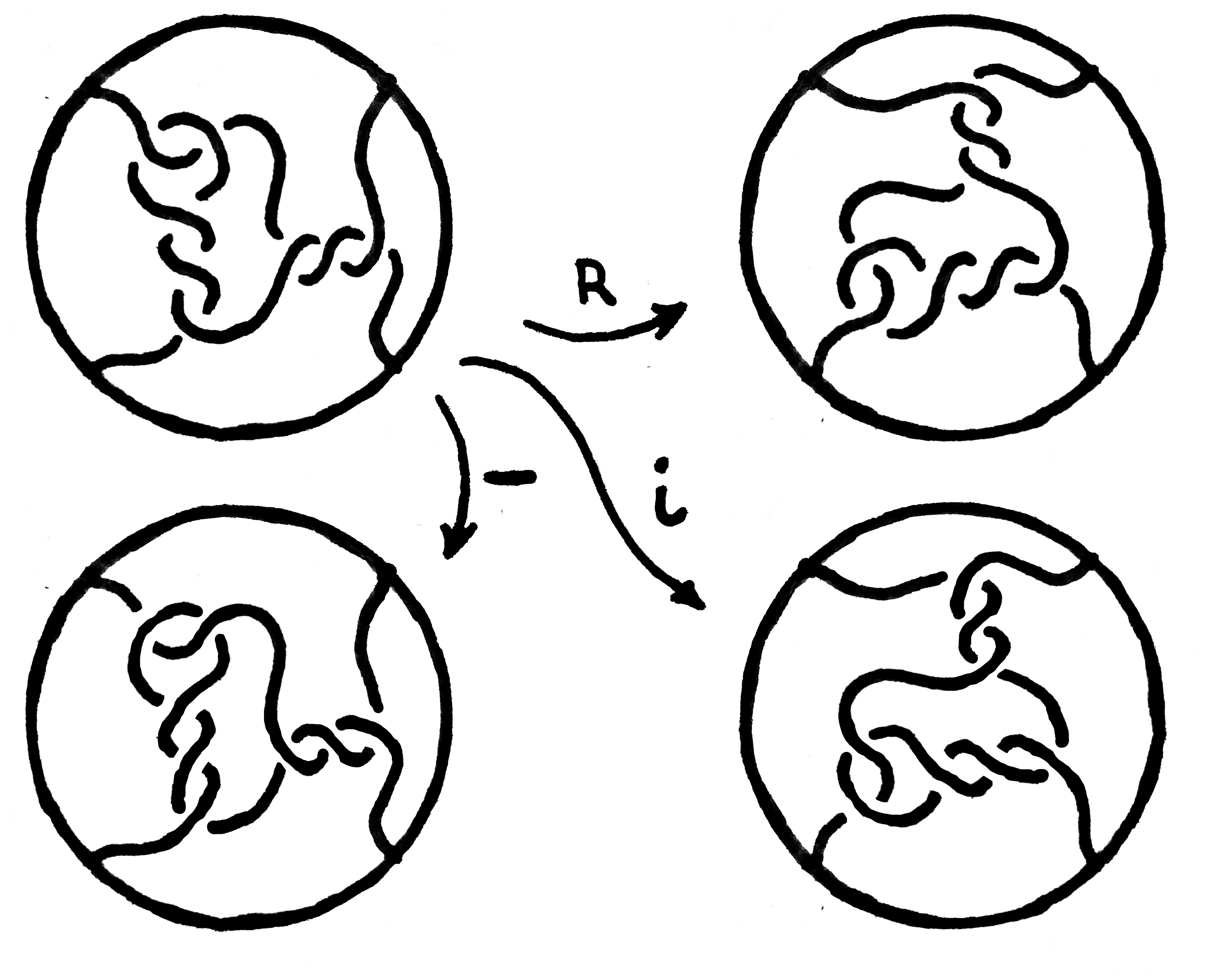}}
\caption{\ti{Mirror image, rotate, inverse}} 
\label{Fig07}
\end{figure}

It immediately follows from Lemma~\ref{lem:ACF} and Lemma~\ref{lem:OP} that any rational tangle can be represented in the so-called "continued fraction form", which means that for any rational tangle $T$ there is an odd number $n$ and there are $a_1\in\mathbb{Z}$ and~$a_2,\,\dots,\,a_{n}\in \mathbb{Z}\setminus\{0\}$ such that the $a_i$’s are all positive or all negative and
$$ T=[a_1]+\frac{1}{[a_2]+\dots+\frac{1}{[a_{n-1}]+\frac{1}{[a_n]}}}.$$
A rational number that corresponds to a continued fraction of the continued fraction form of $T$ is called, the \emph{fraction} of $T$ and denoted by $\F(T)$, that is, 
$$\F(T)=a_1+\frac{1}{a_2+\dots+\frac{1}{a_{n-1}+\frac{1}{a_n}}}$$
if $T\neq [\infty]$ and $\F([\infty])=\infty=\frac{1}{0}$ as a formal expression. And we have the following theorem.
\begin{thm}[Classification of rational tangles (\cite{KL02})]
Two rational tangles are isotopic if and only if they have the same fraction.
\end{thm}

The \emph{numerator} of a $2$-tangle $T$ is a link obtained as the uni\-on of the strings of~$T$ and two shortest simple arcs, one of which connects the two upper endpoints of~$T$, and the other connects the two lower endpoints of~$T$, see Figure~\ref{Fig08}. We denote the numerator of~$T$ by~$\N(T)$. A knot $K$ is called \emph{rational} if there is a rational tangle $T$ such that $K=\N(T)$. And we have the following theorem.

\begin{thm}[Classification of rational knots (\cite{KL02})]
Let $T$ and $S$ be two rational tangles such that 
$$\F(T)=\frac{p}{q},\,\,\,\,\,\,\,\,\,\,\,\F(S)=\frac{l}{m},$$ where $p$ and $q$ are relatively prime, as are $l$ and $m$. Then rational knots $\N(T)$ and~$\N(S)$ are ambient isotopic if and only if $p=l$ and either $q\equiv m\,(\Mod\,p)$ or~$qm\equiv1\,(\Mod\,p)$.
\end{thm}

\begin{figure}[H]
\center{\includegraphics[width=0.7\textwidth]{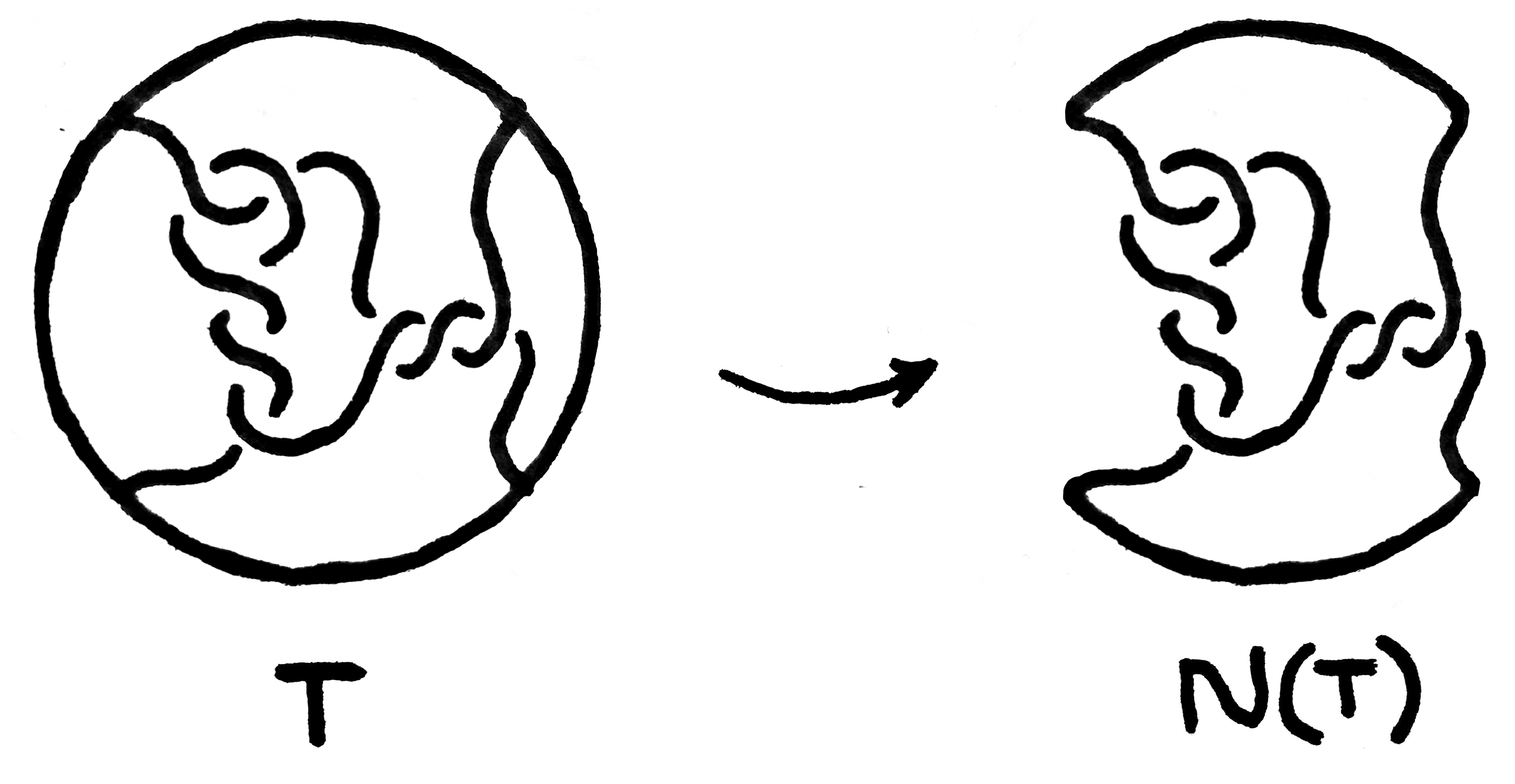}}
\caption{\ti{Numerator}} 
\label{Fig08}
\end{figure}

We also need the following important technical lemma.

\begin{lemma}[Knot or link (\cite{Cro04})]
\label{lem:odev}
Let $T$ be a rational tangle such that $\F(T)=p/q$, where $p$ and $q$ are relatively prime. Then $\N(T)$ is a knot if $p$ is odd, and a link if~$p$ is even.
\end{lemma}
\section{Gordian graphs}\label{sec:gordian}
In this section, we define the notion of a Gordian graph for a knot transformation and introduce some related notation. In addition, we introduce four new (global) knot transformation invariants that reflect the behavior of the corresponding Gordian graph at infinity. In a sense, these invariants can be considered a modification of the classical notion of ends of a topological space, see~\cite{Fre31}.

\begin{definition}[Gordian graph]
Let $\delta$ be a knot transformation, and let $M$ be a subset of $\bL$. Then we denote by $\G(\delta,M)$ such a graph whose  vertex set is in one-to-one correspondence with $M$, and two vertices are connected by an edge if and only if the corresponding links are obtained from each other (sic!) by a single application of $\delta$. The graph $\G(\delta,M)$ is called the \emph{Gordian graph} for $\delta$ and $M$. We consider this graph with a natural metric.
The distance between two vertices $K$ and $Q$ in this metric is denoted by $\dt_{\delta}^{M} (K,Q)$ and is called the \emph{Gordian distance}. Note that $\dt_{\delta}^{M} (K,Q)$ is the number of edges in any shortest path connecting $K$ and $Q$, if $K$ and $Q$ are in the same connected component. If $K$ and $Q$ lie in distinct connected components of $\G(\delta,M)$, then we assume that $$\dt_{\delta}^{M}(K,Q)=\infty.$$
Further, we neglect the difference between a vertex of a Gordian graph and its corresponding link, identifying these objects.
\end{definition}

\begin{definition}[unknotting distance, balls and spheres]
\label{ballandsphere}
Let $\delta$ be a knot transformation, and let $M$ be a subset of $\bL$. For a link $L\in M$ we define the \emph{unknotting distance} as the distance from $L$ to the unknot and denote it by
$$\un_{\delta}^{M}(L)=\dt_{\delta}^{M}(L,\U).$$

If $n\in\mathbb{N}$ and $K\in M$ then we denote by ${\Sf}_{n,\delta}^{M}(K)$ the set of all $Q\in M$ such that~$\dt_{\delta}^{M}(K,Q)=n$, that is, the sphere in $\G(\delta,M)$ of radius $n$ centered at $K$. Also we denote by ${\B}_{n,\delta}^{M}(K)$ the set of all $Q\in M$ such that $\dt_{\delta}^{M}(K,Q)\le n$.
\end{definition}

\begin{remark}
In what follows, we mainly consider either $\bK$ or $\bL$ as $M$, so for these cases we use slightly less-cluttered notation. We omit the reference to $M$ in the case of $M=\bK$, that is, we denote 
$$  \dt_{\delta}^{\bK}(K,Q)\,\, by\,\, \dt_{\delta}(K,Q),\,\,\,\,\,\,\,\, \un_{\delta}^{\bK}(K)\,\, by\,\, \un_{\delta}(K),$$ 
$${\Sf}_{n,\delta}^{\bK}(K)\,\,\, by\,\,\, {\Sf}_{n}^{\delta}(K),\,\,\,\,\,\,\,\,\,\, {\B}_{n,\delta}^{\bK}(K)\,\,\, by\,\,\, {\B}_{n}^{\delta}(K),   $$
for $K,Q\in \bK$, and we replace the reference to $M$ by $\circ$ in the case of $M=\bL$, that is, we denote 
$$  \dt_{\delta}^{\bL}(K,Q)\,\, by\,\, {\ci{\dt_{\delta}}(K,Q)},\,\,\,\,\,\,\,\, \un_{\delta}^{\bL}(K)\,\, by\,\, \ci{\un_{\delta}}(K),$$ 
$${\Sf}_{n,\delta}^{\bL}(K)\,\,\, by\,\,\, {\ci{\Sf_{n}^{\delta}}}(K),\,\,\,\,\,\,\,\,\,\, {\B}_{n,\delta}^{\bL}(K)\,\,\, by\,\,\, \ci{{\B}_{n}^{\delta}}(K),   $$ for $K,Q\in \bL$.

Further, when it comes to removing a certain set of vertices \mbox{$X\subset \Vrt(G)$} from a graph $G$, we always mean that all edges adjacent to these vertices are also removed. For the resulting space we use the same notation as for the ordinary complement, that is, $G\sm X$, but we never mean the ordinary complement in this context. 
\end{remark}

\begin{definition}[equivalence]
We say that two knot transformations $\delta$ and $\zeta$ are \emph{equivalent} if $\G(\delta,\bL)$ coincides with $\G(\zeta,\bL)$.
\end{definition}

We now introduce four indicators of graph behavior at infinity. Definitions of these indicators are based on a general principle but differ in detail. 

\begin{definition}[ends]
Let $G$ be a connected finite or countable graph (in what follows, all occurring graphs are assumed to be finite or countable), and let $\dt_G$ be the natural metric on $G$. Then we define \mbox{$\BU(G), \BI(G), \FU(G), \FI(G)\in\{0\}\cup\mathbb{N}\cup\{\infty\}$} as follows:
$$\BU(G)=\sup_{V\,{ \in }\,{\PB}(G)} \UCC\Big(G\sm V\Big),$$

$$\BI(G)=\sup_{V\,{ \in }\,{\PB}(G)} \ICC\Big(G\sm V\Big),$$

$$\FU(G)=\sup_{V\,{ \in }\,{\PF}(G)} \UCC\Big(G\sm V\Big),$$

$$\FI(G)=\sup_{V\,{ \in }\,{\PF}(G)} \ICC\Big(G\sm V\Big),$$
where $\UCC(G\sm V)$ is the cardinality~of the set of unbounded con\-nec\-ted components of $G\sm V$, \mbox{$\ICC(G\sm V)$} is the cardinality of the set of infinite connected components of $G\sm V$, ${\PF}(G)$ is the set of all finite subsets of $\Vrt(G)$, and ${\PB}(G)$ is the set of all bounded (as subsets of $G$) subsets of $\Vrt(G)$. 

Let $x$ be an arbitrary vertex of $G$, then it is easy to see that $\BU(G)$ can be reformulated as follows
$$\BU(G)=\sup_{n\,{ \in }\,\mathbb{N}} \UCC\Big(G\sm\B_{n}^G(x)\Big),$$
where $\B_{n}^G(x)$ is the set of all $v\in\Vrt(G)$ such that $\dt_G(v,x)\le n$. And in fact the construction does not depend on the choice of $x$. 

We can extend our definitions to the case of graphs with more than one connected component, taking as a value the sum of the corresponding values over all connected components. Let $G$ be a graph, then we say that $\BU(G)$ is the number of \mbox{\emph{$\BU$-ends}}, $\BI(G)$ is the number of \emph{$\BI$-ends}, $\FU(G)$~is the number of \emph{$\FU$-ends}, and $\FI(G)$ is the number of \emph{$\FI$-ends} of $G$.
\end{definition}

\begin{remark}
    Note that for any graph $H$ we have 
$$\FU(H)\le\BU(H)\le\BI(H)\,\, and\,\,\FU(H)\le\FI(H)\le\BI(H),$$ and all inequalities can be strict (for example, let $X_i$ be copies of an infinite complete graph with $i\in\mathbb{Z}$, and let each vertex $x\in X_i$ be connected by edges with its copies in $X_{i-1}$ and $X_{i+1}$. Then the graph $A$ obtained by gluing an infinite complete graph to an arbitrary vertex of this graph has $\FU(A)=1$, $\FI(A)=2$, $\BU(A)=2$, and $\BI(A)=\infty$).
\end{remark}

\section{Local moves}\label{sec:local}
In this section, we define the most well-studied and natural class of knot transformations. Knot transformations included in this class are called local moves. Note that often in works devoted to local moves, the definitions are not strict and partly rely on intuition, which is why the narrative acquires many subtle and unclear moments. To avoid confusion, we give a complete formal definition of a local transformation, which, nevertheless, fully corresponds to intuition. In addition, we give some examples of well-known local moves and families of local moves. Also, we introduce a new family of local moves called almost trivial moves. 

\begin{definition}[local-move-pattern]
Let $B$ be a $3$-ball in $S^3$. A \emph{local-move-pattern} is a {pair $(B_A,B_C)$}, where $B_A$ and $B_C$ are tangles such that $\partial A=\partial C$.
\end{definition}

\begin{definition}[local move]
Let $L_1$ and $L_2$ be two links in $S^3$, and let $B$ be \mbox{a~$3$-ball} in $S^3$, and let $\delta$ be a local-move-pattern with  the base ball $B$. Then we say that~$L_2$ is obtained from $L_1$ by a \emph{local $\delta$-move}, if there are links $L_1^\prime$ and~$L_2^\prime$ in~$S^3$ such that~$L_1^\prime$ is ambient isotopic to $L_1$, $L_2^\prime$ is ambient isotopic to $L_2$, $L_1^\prime$ and~$L_2^\prime$ coincide outside the interior of $B$ in $S^3$, and the pair $((B,B\cap L_1^\prime),(B,B\cap L_2^\prime))$ coincides with $\delta$.
\end{definition}

\begin{remark}
In what follows, we use the same notation for a local-move-pattern and for its corresponding local move as a knot transformation. For example, if $\delta$ is a local-move-pattern, then we denote by $\G(\delta,M)$ the Gordian graph for a local~\mbox{$\delta$-move} and $M\subset\bL$.
\end{remark}

\begin{definition}[$\X$-move (\cite{Wen37, Sch85, BCJTT2017})]
The local move defined by the pattern shown in Figure~\ref{Fig01} is called the \emph{$\X$-move}.
\end{definition}

\begin{remark}
For convenience, we show a local-move-pattern as two balls with strings, bearing in mind that this is the same ball with two sets of strings. In other words, we "highlight" first one and then another set of strings in the same ball.
\end{remark}

\begin{definition}[$\Delta$-move (\cite{MN89, O90, Hor08})]
The local move defined by the pattern shown in Figure~\ref{Fig10} is called the \emph{$\Delta$-move}.
\end{definition}

\begin{figure}[H]
\center{\includegraphics[width=0.6\textwidth]{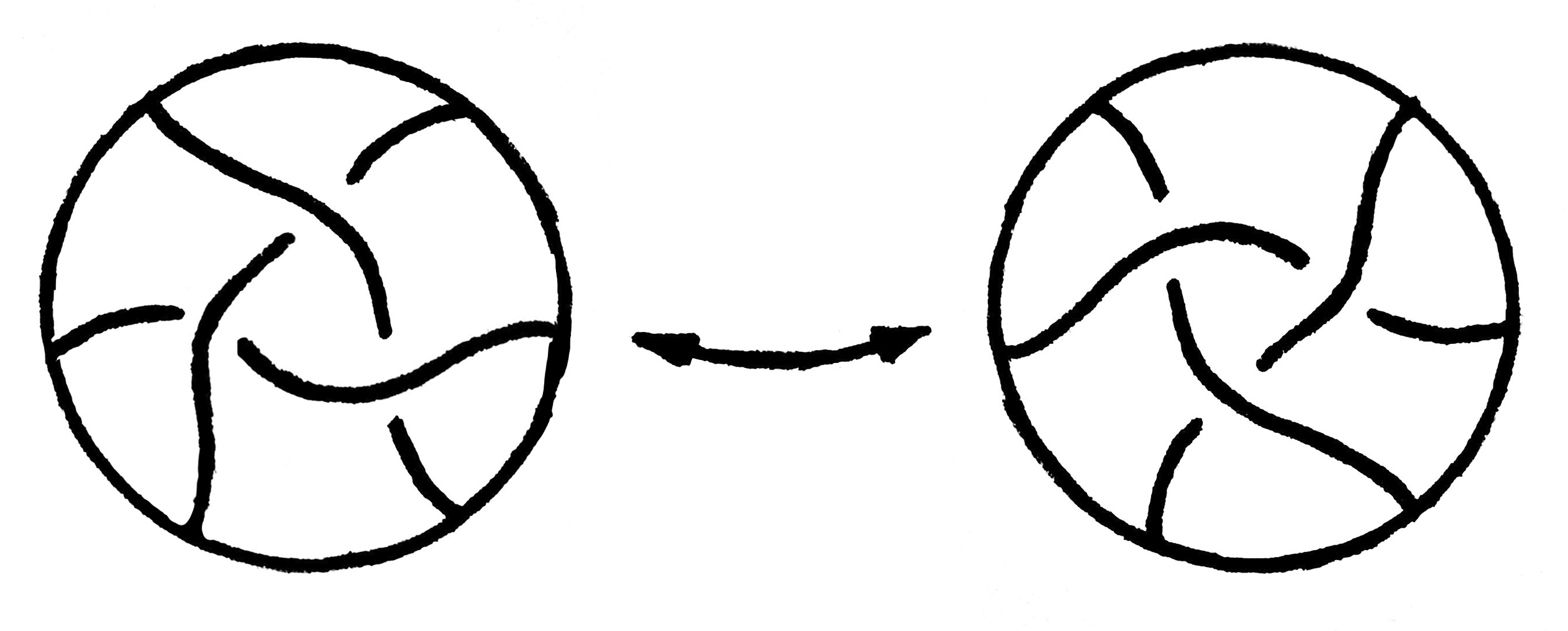}}
\caption{\ti{$\Delta$-move}} 
\label{Fig10}
\end{figure}

\begin{definition}[Clasp-pass-move (\cite{TY02})]
The local move defined by the pattern shown in Figure~\ref{Fig11} is called the \emph{Clasp-pass-move}.
\end{definition}

\begin{figure}[H]
\center{\includegraphics[width=0.6\textwidth]{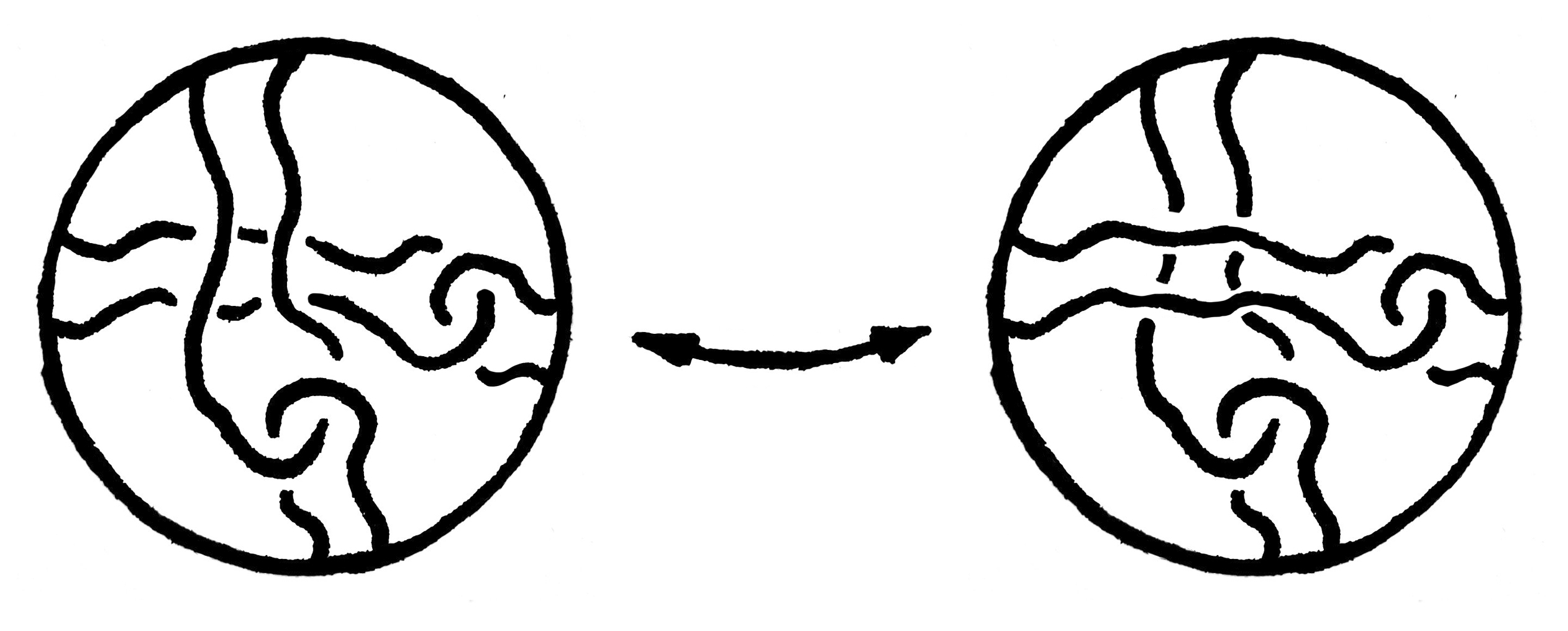}}
\caption{\ti{Clasp-pass-move}} 
\label{Fig11}
\end{figure}

\begin{definition}[$\Cm(n)$-move (\cite{Hab00, OY08, HO12})]
\label{cnmoves}
The local move defined by the pattern shown in Figure~\ref{Fig12} is called the \emph{$\Cm(n)$-move} for $n\in\mathbb{N}$.
\end{definition}

\begin{remark}
Note that the $\X$-move is equivalent to the $\Cm(1)$-move, the $\Delta$-move is equi\-valent to the $\Cm(2)$-move, and the Clasp-pass-move is equivalent to the~\mbox{$\Cm(3)$-move}.
\end{remark}

\begin{figure}[H]
\center{\includegraphics[width=0.6\textwidth]{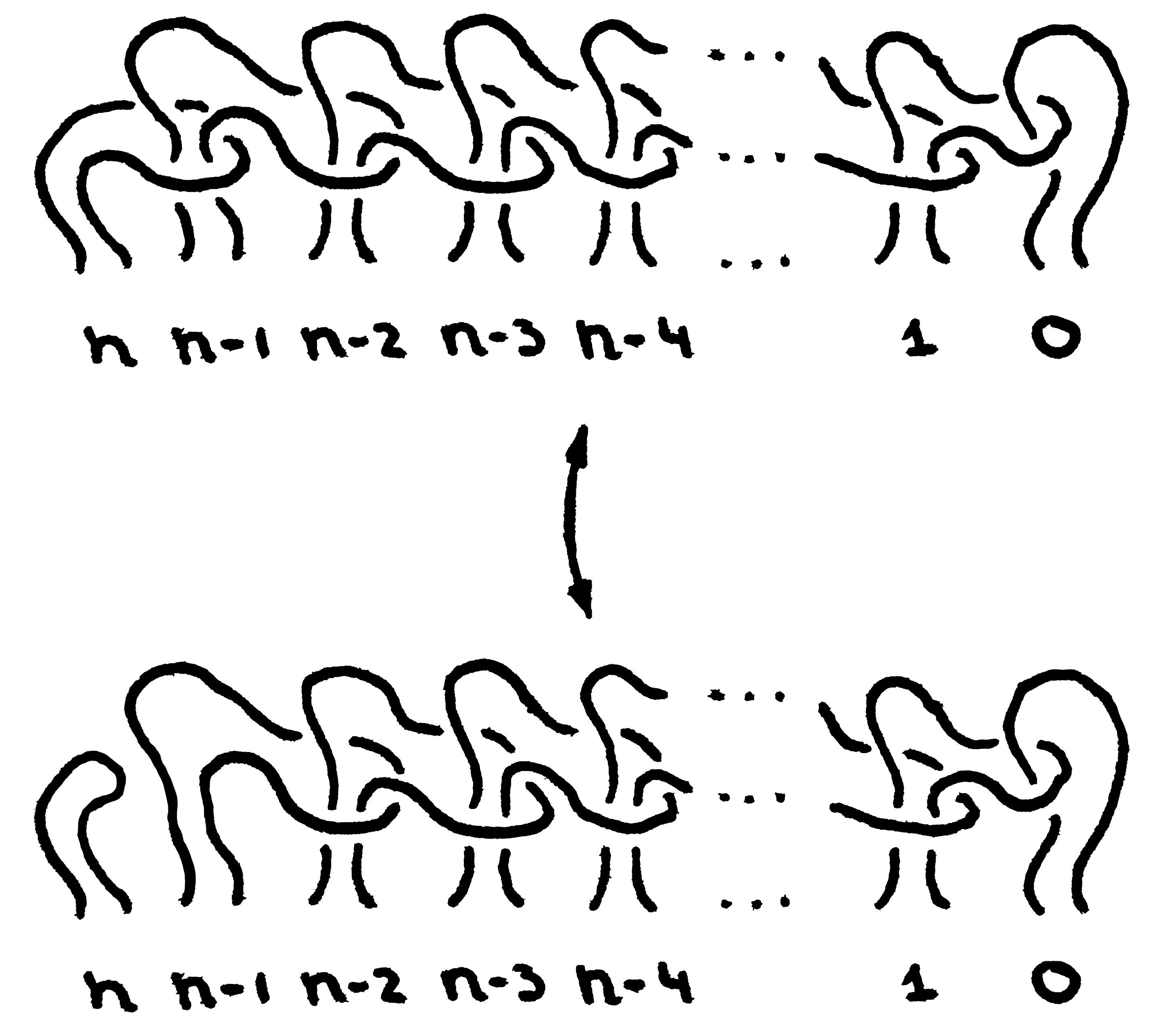}}
\caption{\ti{C(n)-move}} 
\label{Fig12}
\end{figure}

\begin{definition}[$\Hm(n)$-move (\cite{HNT90, KM09, K11})]
\label{hnmoves} 
The local move defined by the pattern shown in Figure~\ref{Fig13} is called the \emph{$\Hm(n)$-move} for $n\in\mathbb{N}$.
\end{definition}

\begin{figure}[H]
\center{\includegraphics[width=0.8\textwidth]{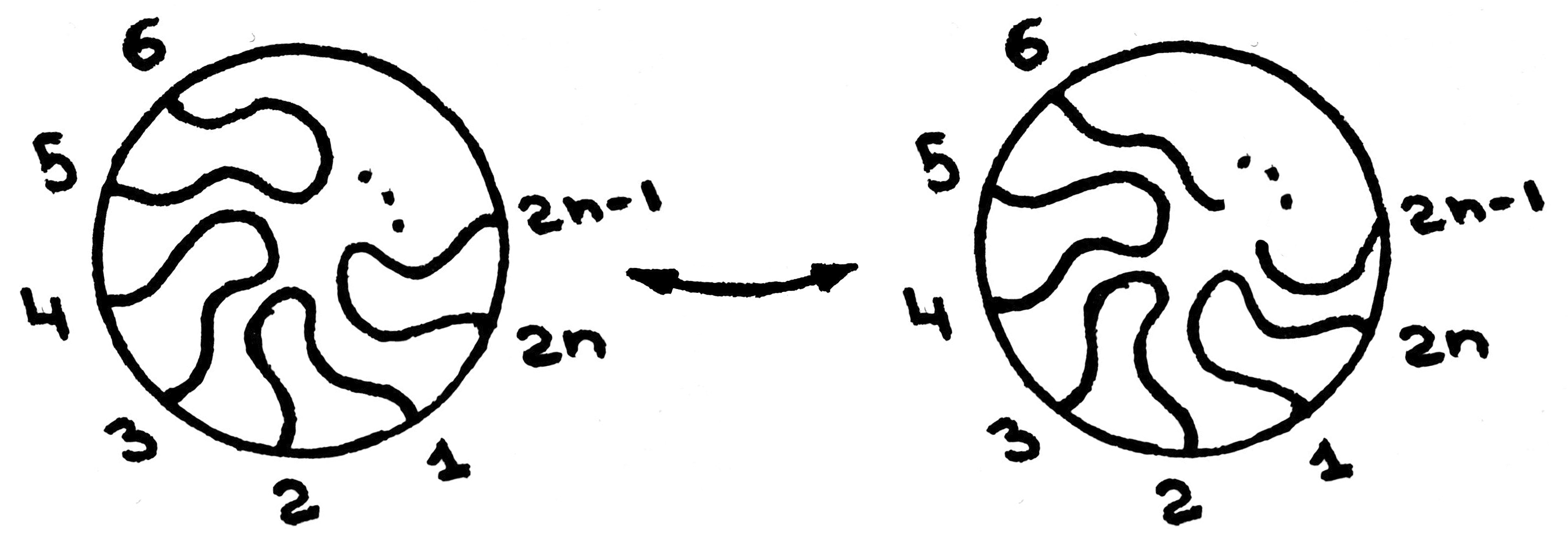}}
\caption{\ti{H(n)-move}} 
\label{Fig13}
\end{figure}

\begin{definition}[$n$-move (\cite{Prz88, DP02, MWY19})]
The local move defined by the pattern shown in Figure~\ref{Fig14} is called the \emph{$n$-move} for $n\in\mathbb{N}$.
\end{definition}

\begin{figure}[H]
\center{\includegraphics[width=0.6\textwidth]{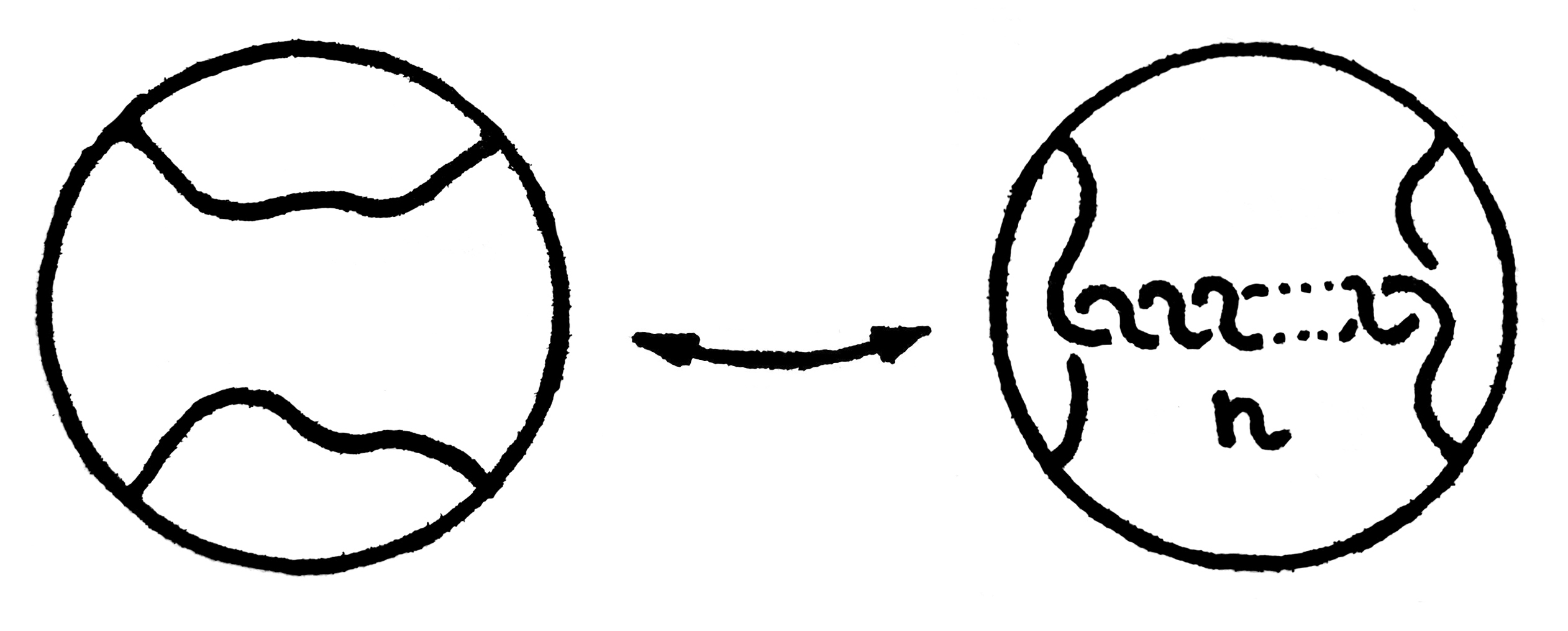}}
\caption{\ti{n-move}} 
\label{Fig14}
\end{figure}

\begin{definition}[Rational $\ra{p}{q}$-moves (\cite{DS00, DP04})] 
\label{rationalmoves} 
A knot transformation $\delta$ is called a \emph{rational move} if there is a rational tangle $T\neq[0]$ such that $\delta$ is equivalent to a local move whose local-move-pattern is $([0],T)$. If $\F(T)=p/q$, then we call $\delta$ the~\mbox{\emph{$\ra{p}{q}$-move}} 
and denote its local-move-pattern by~$\ra{p}{q}$. For example, we denote by~$\G(\ra{p}{q}, M)$ the $\ra{p}{q}$-move Gordian graph, where $M\subset\bL$. 
\end{definition}

\begin{remark}
Note that the $\Hm(2)$-move is equivalent to the $\ra{1}{0}$-move, the \mbox{$\X$-move} is equivalent to the $\ra{2}{1}$-move, and the $n$-move is equivalent to the \mbox{$\ra{n}{1}$-move}. Moreover, it is easy to see that there are equivalent rational moves, with unequal fractions, for example, the \mbox{$\Hm(2)$-move} is also equivalent to the $\ra{1}{1}$-move. 
\end{remark}

\begin{definition}[almost trivial tangle]
An $n$-tangle $B_A$ is called \emph{$n$-string almost trivial} if there is an orientation-preserving homeomorphism of pairs
$$ h\colon (B,A)\rightarrow(D^2\times I,\{x_1,x_2,\dots,x_n\}\times I)$$
where $x_1,x_2,\dots,x_n$ are distinct points on $D^2$. 
\end{definition}

\begin{definition}[almost trivial moves]
\label{almosttrivialmoves} 
A knot transformation $\delta$ is said to be an~{\emph{\mbox{$n$-string} almost trivial move}} if $\delta$ is equivalent to a local move whose local-move-pat\-tern is $(T, S)$, where $T$ and $S$ are $n$-string almost trivial tangles.
\end{definition}

\begin{remark}
Note that all transformations defined above are almost trivial moves.
\end{remark}

\section{Branched covers}\label{sec:branched}
In this section, we recall the concept of a branched covering and construct a cyclic branched covering of the $3$-sphere, branched along a knot, see~\cite{Ro03}, and a cyclic branched covering of a tangle. In addition, we note that the Montesinos trick allows one to obtain estimate for the Gordian distance for any almost trivial move, similar to the result for the $\Hm(n)$-moves in~\cite{HNT90}.

\begin{definition}[branched covering (\cite{Ro03})]
Let $M$ and $N$ be compact $n$-manifolds with $(n-2)$-submanifolds $A\subset M$ and $B\subset N$. Then a continuous func\-ti\-on~$f\colon M\rightarrow N$ is said to be a \emph{branched covering} with branch sets $A$ (upstairs) and $B$ (downstairs) if components of preimages of open sets of $N$ are a basis for the topology of $M$, and $f(A)=B$, $f(M\sm A)=N\sm B$, and $N\sm B$ is exactly the set of points $N$ which are evenly covered, i.e. have neighbourhoods $U$ such that~$f$ sends each component of $f^{-1}(U)$ homeomorphically onto $U$. The restric\-tion $f^\circ\colon M\sm A\rightarrow N\sm B$ is a covering, and, by the compactness of $M$, it is finite-sheeted. Each branch point $x\in A$ has a branching index $p$, meaning that~$f$ is $p$-to-one near $x$, and this number is constant on components of $A$. We call $f$ a~\emph{$p$-fold branched covering} if $f^\circ$ is a $p$-fold covering.
\end{definition}

\begin{definition}[$p$-fold cyclic branched cover of a knot (\cite{Ro03})]
\label{fcbc}
Let $K$ be a knot in~$S^3$, let $M$ be a Seifert surface for $K$, let \mbox{$M^\circ=M\sm K$}, let \mbox{$\N\colon M^\circ\times (-1,1)\rightarrow S^3$} be a bicollar of $M^\circ$. Let 
$$N^+=\N(M^\circ\times(0,1)),\,\,\,\,\,N^-=\N(M^\circ\times(-1,0)),$$
$$N=\N(M^\circ\times(-1,1)),\,\,\,\,\,Y=S^3\sm M,\,\,\,\,\, X=S^3\sm K.$$
Let $(N_i,N_i^{+},N_i^{-})$  be a copy of the triple $(N,N^{+},N^{-})$, and let $(Y_i,N_i^{+},N_i^{-})$ be a copy of the triple $(Y,N^{+},N^{-})$, where \mbox{$i\in\{0,1,\dots,p-1\}$}. Let $$N^\circ=\bigcup_{i=0}^{p-1}N_i\,\,\, and\,\,\,Y^\circ=\bigcup_{i=0}^{p-1}Y_i$$ be the disjoint unions. Now we identify $N_{i}^{+}\subset Y_i$ with \mbox{$N_{i}^{+}\subset N_i$} by the identity homeomorphism, and identify $N_{i}^{-}\subset Y_i$ with $N_{i+1}^{-}\subset N_{i+1}$ in a similar way for each \mbox{$i\in\{0,1,\dots,p-2\}$}, and identify $N_{p-1}^{-}\subset Y_{p-1}$ with $N_{0}^{-}\subset N_{0}$. We denote the resulting space by $X^\circ_p$.
Note that we obtain a $p$-fold cyclic covering, denoted by $f\colon X^\circ_p\rightarrow S^3\sm K $. Let $W$ be an open tubular neighbourhood of $K$ in~$S^3$. Let $f_\circ\colon f^{-1}(S^3\sm W )\rightarrow S^3\sm W$ be a $p$-fold cyclic covering obtained by the corresponding restriction of $f$. Note that $f_\circ^{-1}(\partial(\Cl W)$ is a torus, and~$f_\circ^{-1}(m)$ is a single curve on $f_\circ^{-1}(\partial\Cl W)$, where $m$ is the meridian of $\partial (\Cl W)$. We can attach a solid torus $S^1\times D^2$ to $f^{-1}(S^3\sm W )$ by gluing along the boundary in such a way that the meridian of $S^1\times D^2$ is glued to $f_\circ^{-1}(m)$. We obtain a closed connected orientable \mbox{$3$-manifold} denoted by $\Sigma_p(K)$. We can extend the covering map to a \emph{$p$-fold cyclic branched covering} $\Sigma_p(K)\rightarrow S^3$ by sending $S^1\times D^2$ onto~$\Cl W=S^1\times D^2$ by the product of the maps $z\rightarrow \sfrac{z^p}{|z^{p-1}|}$ on $D^2$ and identity on $S^1$. The branch set is $K$ in $S^3$ downstairs and some knot in $\Sigma_p(K)$ upstairs. 
\end{definition}

\begin{remark}
Let $K$ be a knot in $S^3$. We denote by $\e_p(K)$ the minimum number of generators of the first homology group with integer coefficients of the \mbox{$p$-fold} cyclic branched cover of $S^3$ branched along $K$ denoted by ${\bf H}_1(\Sigma_p(K))$.
\end{remark}

\begin{lemma}[double branched cover of a rational knot (\cite{Kaw96})]
\label{lem:dbc}  
Let $T$ be a rational tangle such that $\F(T)=p/q,$ where $p$ and $q$ are relatively prime. Then the double branched cover $\Sigma_2(\N(T))$ is the lens space $\Ll(p,q)$. Note that then $\e_2(\N(T))=1,$
since in this case ${\bf H}_1(\Sigma_2(\N(T)))=\mathbb{Z}_{p}.$
\end{lemma}

\begin{definition}[$p$-fold cyclic branched cover of a rational tangle]
To begin with, we construct a $2$-fold branched covering for $[\infty]=B_C$ as follows. We cut the ball $B$ as shown in Figure~\ref{Fig15}, after which we glue two copies of the resulting cut ball together as shown in Figure~\ref{Fig15} (to simplify visualization, we further depict tangles in Figures not as balls, but as cubes with strings). Thus we obtain a solid torus with two distinguished branching arcs. Assuming that $f$ is identical on both copies of the cut ball and taking into account the gluing rule it is easy to see that~$f\colon S^1\times D^2\rightarrow B$ is a $2$-fold branched cover of $B$ branched over $C$ downstairs and two arcs upstairs.

We can construct a $2$-fold branched covering of any other rational tangle similarly by properly choosing the two discs along which we need to cut and glue. Figure~\ref{Fig16} shows the corresponding disks for $\frac{1}{[2]}$. A \emph{$p$-fold cyclic branched covering} is also constructed similarly, but we need to glue not two, but $p$ copies of the cut ball in a cyclic order.
\end{definition}

\begin{figure}[H]
\center{\includegraphics[width=0.7\textwidth]{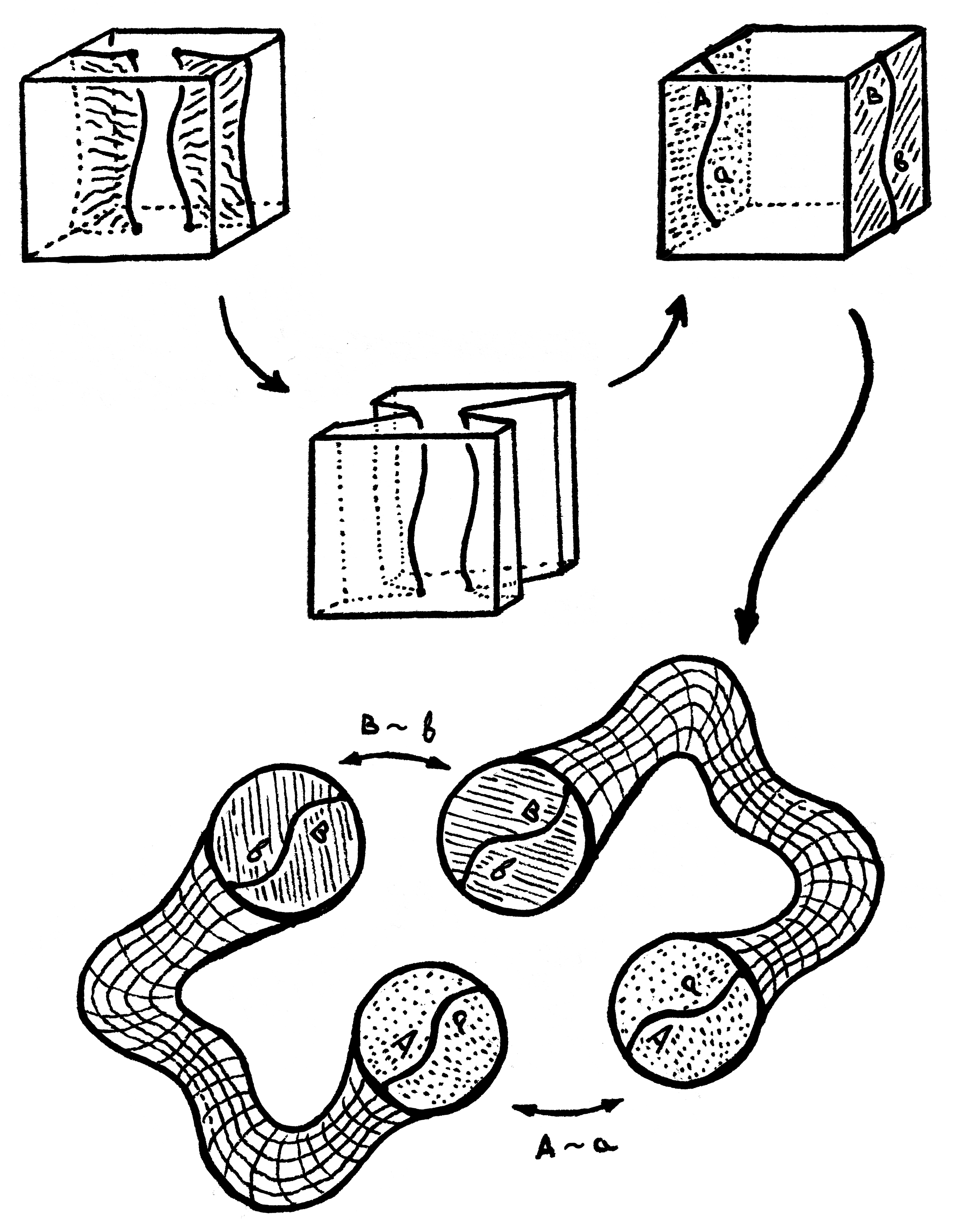}}
\caption{\ti{2-fold branched cover for $[\infty]$}} 
\label{Fig15}
\end{figure}

\begin{namedlem}[The Lower Estimates Lemma]
\label{anothertanglebr}
Let $\delta$ be $n$-string almost trivial move, let $K$ and $Q$ be knots, and let $p$ be integer such that $p\ge 2$. Then we have
$$\dt_{\delta}(K, Q)\ge\frac{|\e_p(K)-\e_p(Q)|}{(n-1)(p-1)}.$$
\end{namedlem}

\begin{proof}
Note that for each rational tangle $B_C$ there is such a disk $D^{\circ}$ embedded in~$B$ that "separates" the strings of this tangle, that is such that $\partial B\cap D^{\circ}=\partial D^{\circ},$ and $D^{\circ}$ does not intersect the arcs of the tangle, and each component of ${B}\sm D^{\circ}$ contains exactly one string of this tangle. Figure~\ref{Fig16} shows a separating disk for~$[\infty]$. It is easy to see that the $2$-fold branched covering of any rational tangle covers the boundary circle of the separating disk of this tangle by two meridians of~$\partial(S^1\times D^2)$. 
On the other hand, the $2$-fold branched covering of a rational tangle~$B_A$ covers the boundary circle of the separating disk of a rational tan\-gle~${B}_C$ as a curve on the sphere $\partial {B}$ by two identical torus knots lying on~$\partial(S^1\times D^2)$. In the $2$-fold branched cover of $B_C$, these two torus knots correspond to meridians and are the boundaries of the meridian disks. Figure~\ref{Fig17} shows a pair of torus knots covering the boundary circle of a separating disk for $\frac{1}{[2]}$ under the $2$-fold branched covering of $[\infty]$. Thus, if a knot in $S^3$ is changed by a rational move, then the $2$-fold branched cover of $S^3$ branched over this knot is changed by appropriate Dehn surgery. 

\begin{figure}[H]
\center{\includegraphics[width=0.7\textwidth]{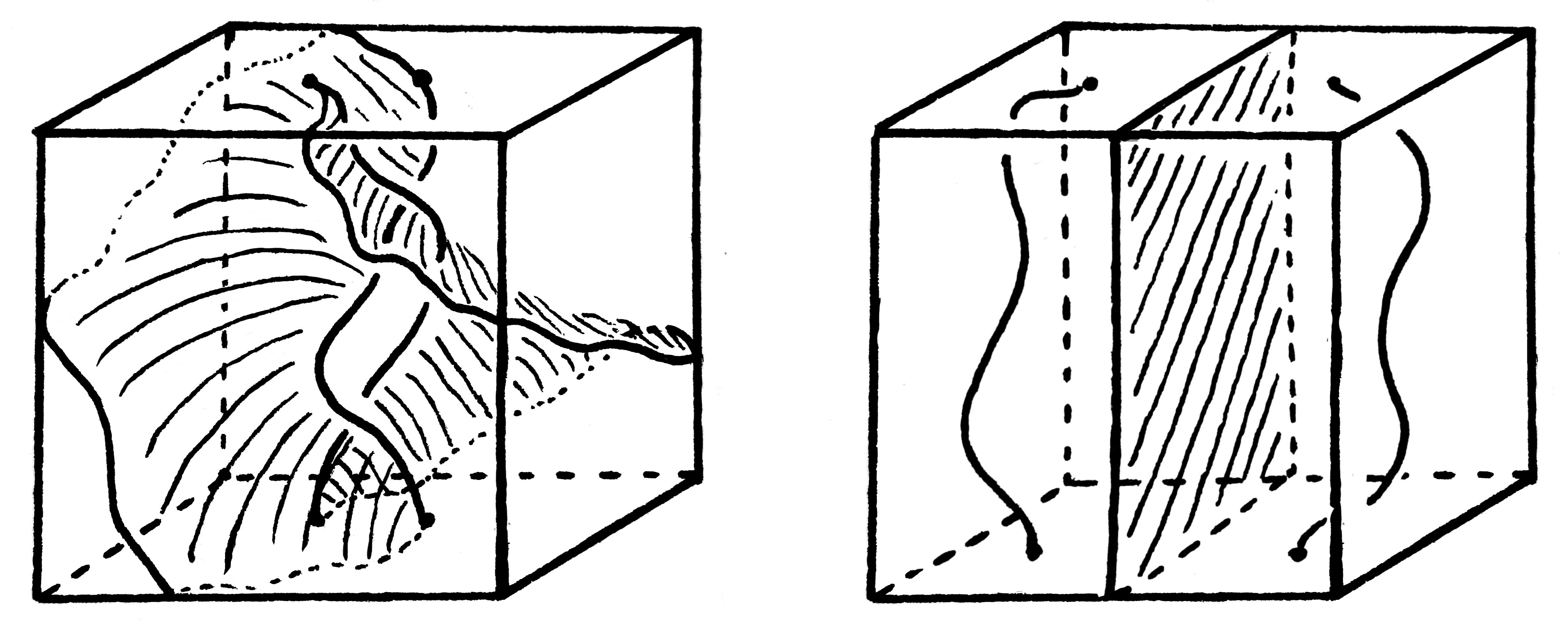}}
\caption{\ti{Disks for $\frac{1}{[2]}$ and the separating disk of $[\infty]$}} 
\label{Fig16}
\end{figure}

It is easy to see that a single Dehn surgery changes the minimum number of generators of the first homology group of a manifold by at most one, and therefore we have $$\dt_{\ra{p}{q}}(K, Q)\ge|\e_2(K)-\e_2(Q)|$$ for any $\ra{p}{q}$-move and any knots $K$ and $Q$. Moreover, if a knot in $S^3$ is changed by an $n$-string almost trivial move, then the \mbox{$p$-fold} branched cover of $S^3$ branched over this knot is changed by appropriate surgery on a handlebody of genus $(n-1)(p-1)$, and therefore we have $$\dt_{\delta}(K, Q)\ge\frac{|\e_p(K)-\e_p(Q)|}{(n-1)(p-1)}$$ for any $n$-string almost trivial move $\delta$ and any knots $K$ and $Q$ (cf.~\cite{HNT90}).
\end{proof}

\begin{figure}[H]
\center{\includegraphics[width=0.8\textwidth]{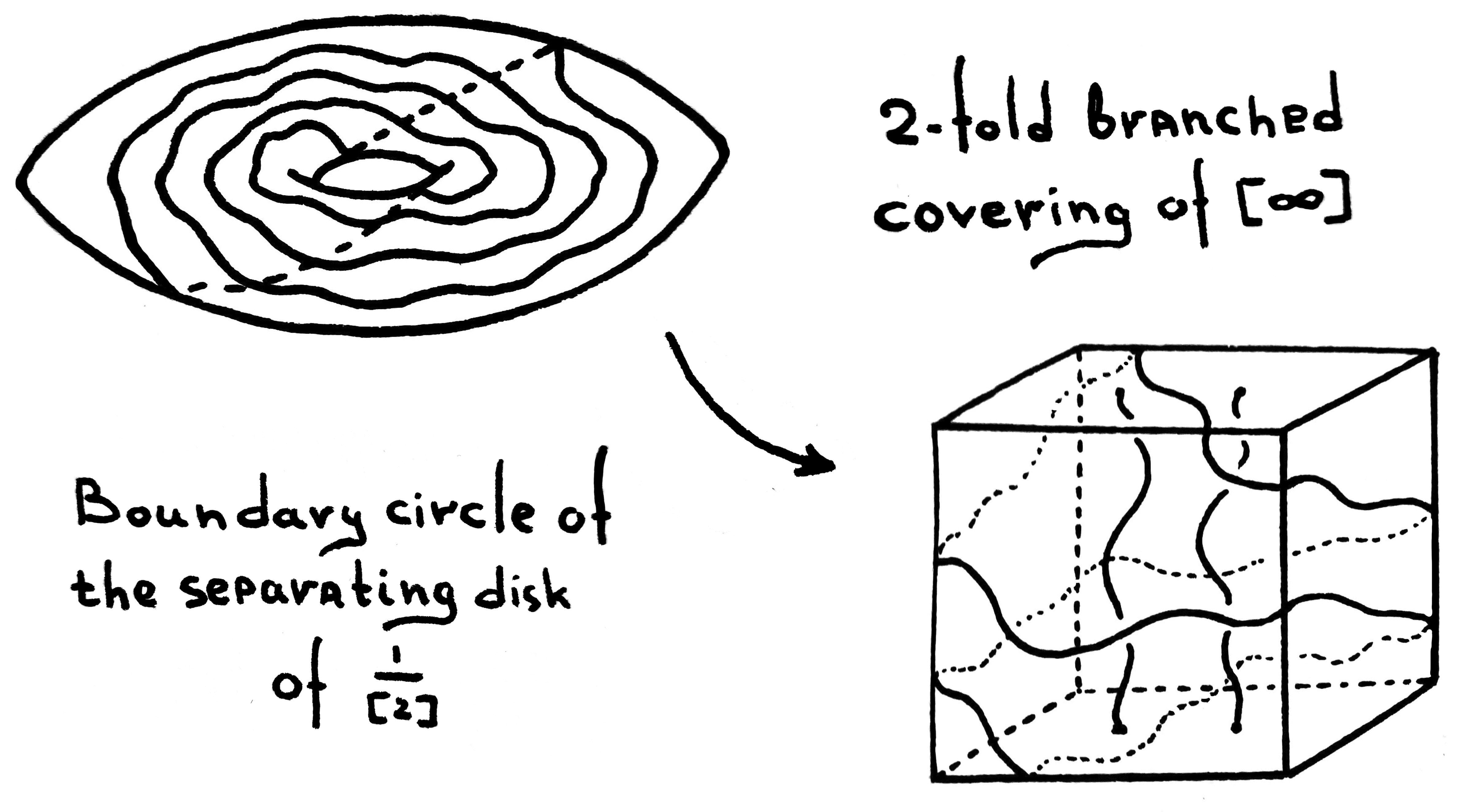}}
\caption{\ti{{The boundary circle of the separating disk of~$\frac{1}{[2]}$ and the covering of $[\infty]$}}} 
\label{Fig17}
\end{figure}

\section{The Alexander polynomial}\label{sec:alexander}
In this section, we recall definitions of two classical polynomial invariants of links, namely the Alexander polynomial, see~\cite{Al28}, and the Conway polynomial, see~\cite{Con70}. We also show the connection of these invariants with each other and with~$2$-fold cyclic branched coverings of $S^3$, branched along knots, see~\cite[p.~206]{Ro03}.

\begin{definition}[Alexander polynomial (\cite{Al28}, \cite{Ro03})]
Let $K$ be a knot in $S^3$. Denote by $X^{\infty}(K)$ the infinite cyclic covering space of the complement of $K$ (its construction is completely similar to the construction of a finite cyclic covering space, except that we need to glue countably many copies of $N_i$ and $Y_i$, identifying $N_i^+\subset Y_i$ with $N_{i}^+\subset N_i$ and identi\-fy\-ing~$N_i^-\subset Y_i$ with $N_{i+1}\subset N_{i+1}$ for each~$i\in\mathbb{Z}$, see construction and notation in Definition~\ref{fcbc}), and denote by $\Lambda$ the ring of finite Laurent polynomials with integer coefficients. Let us define the product of \mbox{$p(t)\in\Lambda$} and $\zeta\in{\bf H}_1(X^\infty(K))$ by the formula
$$p(t)\zeta=c_{-r}\tau_\ast^{-r}\zeta+\dots+c_{-1}\tau_\ast^{-1}\zeta+c_0\zeta+c_1\tau_\ast\zeta+\dots+c_s\tau_\ast^s\zeta,$$
where $\tau\colon X^\infty(K)\rightarrow X^\infty(K)$ is one of the two generators of the group of covering translations, ${\bf \tau}_\ast\colon{\bf H}_1(X^\infty(K))\rightarrow{\bf H}_1(X^\infty(K)) $ is the homology isomorphism induced by $\tau$, and
$$ p(t)=c_{-r}t^{-r}+\dots+c_{-1}t^{-1}+c_0+c_1t+\dots+c_st^s.$$
It is easy to see that this multiplication does not depend on the choice of $\tau$ and defines an $\Lambda$-module structure on ${\bf H}_1(X^\infty(K))$ and this module is called the~\emph{Alexander invariant}. If the order ideal of a presentation matrix for the Alexander invariant is principal then any generator of this ideal is called the \emph{Alexander polynomial}. Note that all these polynomials are equal up to multiplication by a monomial, so we can fix a polynomial with a positive constant term. In what follows, the Alexander polynomial is understood to be precisely this polynomial, which we denote by~$\Delta_K(t)$. Moreover, the order ideal does not depend on the choice of a presentation matrix, the Alexander polynomial always exists and is a knot invariant, see~\cite[p.~207]{Ro03}.
\end{definition}

\begin{lemma}[Alexander polynomial and double branched cover (\cite{Ro03})]
\label{card}
 For any knot~$K$ the group ${\bf H}_1(\Sigma_2(K))$ is finite and $|{\bf H}_1(\Sigma_2(K))|=|\Delta_K(-1)|$. 
\end{lemma}

\begin{definition}[Conway polynomial (\cite{Con70})]
Let $\nabla\colon\bL^\circ\rightarrow\mathbb{Z}[z]$ be a function with the following two properties:

\begin{itemize}
    \item $\nabla(\U^\circ)(z)=1$, where $\U^\circ$ is the oriented unknot;
    \item let $L_+$, $L_-$, $L_0$ be oriented links such that there are oriented links $L_+'$, $L_-'$, and $L_0'$ in $S^3$ and a $3$-ball $B$ in $S^3$ such that $L_+'$, $L_-'$, and $L_0'$ are ambient isotopic to $L_+$, $L_-$, and $L_0$, respectively, $L_+'$, $L_-'$, and $L_0'$ coincide outside the interior of $B$ in $S^3$, and $L_+'$, $L_-'$, and $L_0'$ intersect with $B$ as shown in Figure~\ref{Fig18}. Then we have 
    $$\nabla(L_+)(z)-\nabla(L_-)(z)=z\nabla(L_0)(z);$$ 
\end{itemize}

It is well known, see~\cite{Con70}, that there is only one such function, that is, it is a well-defined invariant of oriented links. If \mbox{$L\in\bL^\circ$}, then $\nabla(L)(z)$ is called the~\emph{Conway polynomial} of $L$. In addition, it is known that the Conway polynomial $\nabla(K)(z)$ does not depend on orientation if $K$ is a knot so it is also a well-defined invariant for unoriented knots. 
\end{definition}

\begin{figure}[H]
\center{\includegraphics[width=0.65\textwidth]{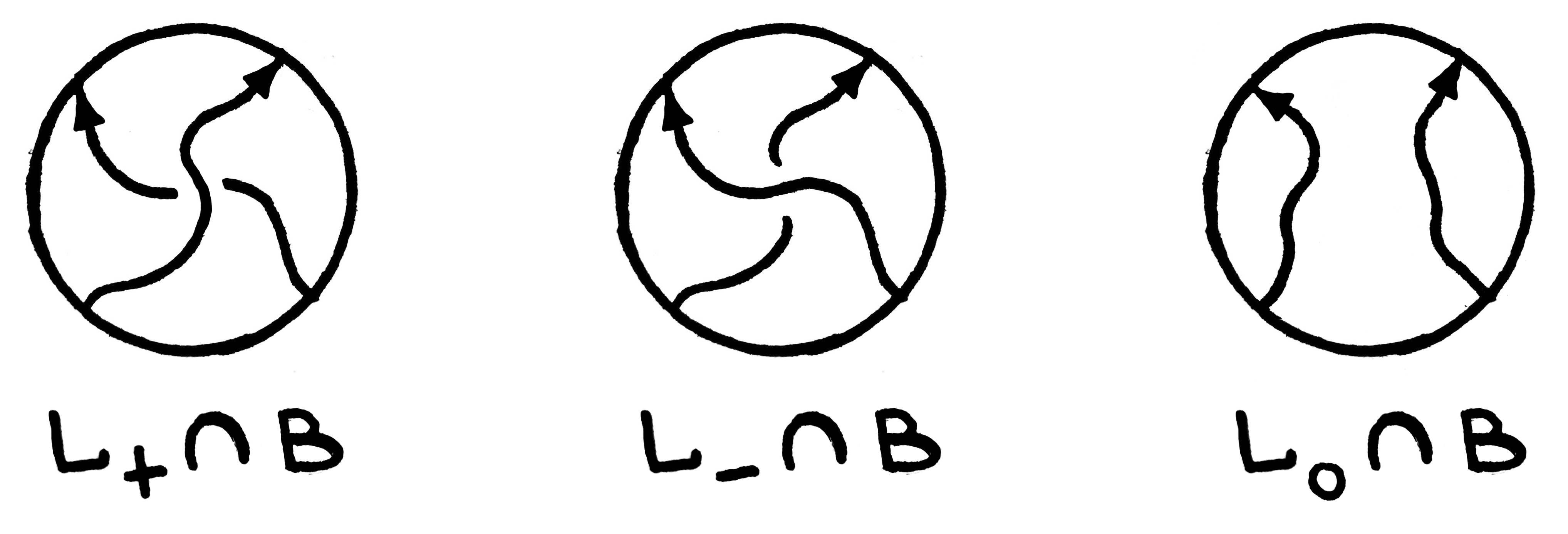}}
\caption{\ti{Intersection with $B$}} 
\label{Fig18}
\end{figure}

\begin{lemma}[Changing the variable (\cite{Con70})]
\label{changev}
Let $K$ be a knot in $S^3$. Then we have $$\Delta_K(t^2)=\nabla(K)(t-t^{-1}).$$
\end{lemma}

\section{Proofs of the Main Theorems}\label{sec:main}
In this section, we state and prove the Path Shifting Lemma, state and prove the Basic Lemma, and prove Theorem~\ref{thm1}, Theorem~\ref{thm2}, Theorem~\ref{thm3}, and Theorem~\ref{thm4}.

\begin{namedlem}[The Path Shifting Lemma] 
\label{shift}
Let $\delta$ be a local move. Let $W$ be an arbitrary knot in $S^3$, and let $\gamma$ be a path in $\G(\delta,\bK)$ such that $\{V_0,V_1,\dots,V_n\}$ is the set of vertices of $\gamma$, $\{e_1,e_2,\dots,e_n\}$ is the set of edges of $\gamma$, and $e_i$ is incident to~$V_i$ and $V_{i-1}$, where $i\in\{1,2,\dots,n\}$. Then there is a path $\gamma'$ in $\G(\delta,\bK)$, such that~$\{V_0',V_1',\dots,V_n'\}$ is the set of vertices of $\gamma'$, $\{e_1',e_2',\dots,e_n'\}$ is the set of edges of $\gamma'$, $V_i'$ is a connected sum of $V_i$ and $W$, and $e_i'$ is incident to $V_i'$ and $V_{i-1}'$, where~$i\in\{1,2,\dots,n\}$. In other words, any path in a local-move Gordian graph can be "shifted" by some given knot.
\end{namedlem}

\begin{proof}
First, we prove the Lemma in the case of a single-edge path. Let $K$ and~$Q$ be knots in $S^3$ corresponding to adjacent vertices of $\G(\delta, \bK)$, and let $W$ be an arbitrary knot in $S^3$. 
Let $B$ be a $3$-ball in $S^3$, and let $B$ be the base ball of $\delta$. By definition, there are knots $K'$ and $Q'$ such that $K'$ is ambient isotopic to $K$, and~$Q'$ is ambient isotopic to $Q$, $K'$ and $Q'$ coincide outside the interior of $B$ in~$S^3$, and the pair $((B,B\cap K'),(B,B\cap Q'))$ coincides with $\delta$. Let $E$ be a $3$-ball in $S^3$ such that $E\cap B=\varnothing$, $\partial E\cap K'$ is two points and $E\cap K'$ is an unknotted arc connecting these two points. Let $E'$ be a $3$-ball in $S^3$ such that $\partial E'\cap W$ is two points and $E'\cap W$ is an unknotted arc connecting these two points. Denote by $S'$ the adjunction space $$\Big(S^3\sm\Int(E)\Big)\bigsqcup_{\ff} \Big(S^3\sm\Int(E')\Big),$$ where $\ff\colon\partial E\rightarrow\partial E'$ is a glueing homeomorphism such that \mbox{$\ff(\partial E\cap K')=\partial E'\cap W$.} Then we introduce the notation
$$K_1=\Big(K'\cap\big(S^3\sm\Int(E)\big)\Big)\cup\Big(W\cap\big(S^3\sm\Int(E')\big)\Big)\subset S'\cong S^3, $$
$$Q_1=\Big(Q'\cap\big(S^3\sm\Int(E)\big)\Big)\cup\Big(W\cap\big(S^3\sm\Int(E')\big)\Big)\subset S'\cong S^3. $$
By construction, $K_1$ is a connected sum of $K$ and $W$, and $Q_1$ is a connected sum of $Q$ and $W$. Moreover, $K_1$ and $Q_1$ coincide outside the interior of $B$, and the pair $((B,B\cap K'),(B,B\cap Q'))$ coincides with $\delta$, which implies that $K_1$ and $Q_1$ correspond to adjacent vertices of $\G(\delta, \bK)$. This argument completes the proof in this case.

We now prove the Lemma in the general case. Let $\gamma$ be a path in $\G(\delta,\bK)$ such that $\{V_0,V_1,\dots,V_n\}$ is the set of vertices of $\gamma$, $\{e_1,e_2,\dots,e_n\}$ is the set of edges of $\gamma$, and $e_i$ is incident to $V_i$ and $V_{i-1}$, where $i\in\{1,2,\dots,n\}$. The Lemma in the case of a single-edge path says that there are adjacent vertices $V_0'$ and $V_1'$ such that $V_0'$ is a connected sum of $V_0$ and~$W$, and $V_1'$ is a connected sum of $V_1$ and~$W$, and there are adjacent vertices $V_1''$ and $V_2'$ such that $V_1''$ is a connected sum of~$V_1$ and $W$, and $V_2'$ is a connected sum of $V_2$ and $W$. Since a connected sum of unoriented knots is not uniquely defined (see Remark after Definition~\ref{consum}), it may be that $V_1'$ and $V_1''$ are not ambient isotopic, but the control over the choice of the gluing homeomorphism in the proof of the Lemma in the case of a single-edge path allows us to assume that the choice of the corresponding connected sums is consistent, that is, $V_1'$ coincides with $V_1''$. Similarly, we can choose each next pair of vertices $(V_{i-1}', V_{i}')$ in such a way that the choice of $V_{i-1}'$ is consistent with the previous choice that defines the pair $(V_{i-2}', V_{i-1}')$. 
This argument completes the proof.
\end{proof}


\begin{namedlem}[The Basic Lemma] 
\label{basiclem}
Let $\delta$ be an $n$-string almost trivial move. Suppose that there is a knot $Q$ such that $Q\in {\Sf}_1^\delta(\U)$ and $\e_2(Q)\geq1$. Then the number of $\BU$-ends of each connected component of $\G(\delta, \bK)$ is equal to one.
\end{namedlem}

\begin{proof}
Let $C$ be a connected component of $\G(\delta,\bK)$, and let $E$ be an arbitrary vertex in $C$. Let us show that for any $r\in\mathbb{N}$ and any knots $K$ and $S$ such that 
$$nr+e\leq\dt_{\delta}(K,E)<\infty\,\,\,and\,\,\,nr+e\leq\dt_{\delta}(S,E)<\infty,$$
where $e=\e_2(E)$, there is a path in $\G(\delta,\bK)$ that connects $K$ and $S$ and does not intersect $\B_{r-1}^{\delta}(E)$. This immediately implies that the number of $\BU$-ends of $C$ is exactly one, since for any $r\in\mathbb{N}$ the complement $C\setminus \B_{r-1}^{\delta}(E)$ contains only one unbounded connected component containing the set $C\sm \B_{nr+e}^{\delta}(E)$. 

\medskip
Let $r\in\mathbb{N}$, and let $K$ be a knot such that \mbox{$nr+e\leq\dt_{\delta}(K,E)<\infty,$} and let~$e_1$ be an edge incident to vertices $\U$ and $Q$. \nameref{shift} says that~$e_1$ can be "shifted" by $Q$, that is, there are a knot $Q^2$ and an edge $e_2$, such that~$Q^2$  is a connected sum of $Q$ and $Q$, and $e_2$ is incident to $Q$ and $Q^2$. Now $e_2$ can be "shifted" in a similar way, that is, there are knots $Q^3$ and $J$, and an edge~$e_3$, such that $Q^3$  is a connected sum of $Q^2$ and~$Q$, $J$ is a connected sum of $Q$ and~$Q$, and $e_3$ is incident to $J$ and $Q^3$. Moreover, as in the proof of \nameref{shift}, we can choose a gluing homeomorphism in such a way that $J$ is ambient isotopic to~$Q^2$. Thus, $e_1$, $e_2$, and $e_3$ form a path in $\G(\delta,\bK)$ connecting $\U$ and $Q^3$. Continuing to "shift" the edges in the same way so that they are consistent with each other, we obtain a path $\gamma$ in $\G(\delta,\bK)$ connecting~$\U$ and~$Q^{(n-1)r+e}$ such that $\{\U=Q^0,Q=Q^1,Q^2,\dots,Q^{(n-1)r+e}\}$ is the set of vertices of $\gamma$, $\{e_1,e_2,\dots,e_{(n-1)r+e}\}$ is the set of edges of $\gamma$, $Q^i$ is a connected sum of~$Q^{i-1}$ and $Q$, and $e_i$ is incident to $Q^i$ and $Q^{i-1}$, where $i\in\{1,2\dots,(n-1)r+e\}$.

\medskip
\nameref{shift} says that $\gamma$ can be "shifted" by $K$, that is, there is a path $\gamma'$ in $\G(\delta,\bK)$, such that $\{K=W_0,W_1,\dots,W_{(n-1)r+e}\}$ is the set of vertices of $\gamma'$, $\{e_1',e_2',\dots,e_{(n-1)r+e}'\}$ is the set of edges of $\gamma'$, $W_i$ is a connected sum of $Q^i$ and $K$, and $e_i'$ is incident to $W_i$ and $W_{i-1}$, where $i\in\{1,2\dots,(n-1)r+e\}$.

\medskip
Note that since $\dt_{\delta}(K, E)<\infty$, there is a path $\zeta$ in $\G(\delta,\bK)$ connecting $E$ and~$K$, such that $\{E=V_0,V_1,\dots,K=V_{m}\}$ is the set of vertices of $\zeta$, $\{f_1,f_2,\dots,f_{m}\}$ is the set of edges of $\zeta$, and $f_i$ is incident to $V_i$ and $V_{i-1}$, where $i\in\{1,2\dots,m\}$. \nameref{shift} says that $\zeta$ can be "shifted" by $Q^{(n-1)r+e}$, that is, there is a path $\zeta'$, such that \mbox{$\{V_0',V_1',\dots,V_{m}'\}$} is the set of vertices of $\zeta'$, $\{f_1',f_2',\dots,f_{m}'\}$ is the set of edges of $\zeta'$, $V_i'$ is a connected sum of $V_i$ and $Q^{(n-1)r+e}$, and $f_i'$ is incident to $V_i'$ and $V_{i-1}'$, where \mbox{$i\in\{1,2\dots,m\}$}.

\medskip
Note that we can choose the gluing homeomorphisms of the "shifts" in such a way that $W_{(n-1)r+e}$ and $V_m'$, each of which is a connected sum of $K$ and $Q^{(n-1)r+e}$, are ambient isotopic. In this case, paths $\gamma'$ and $\zeta'$ form a path connecting $K$ and~$V_0'$.

\medskip
Now we need some technical details. 
\nameref{anothertanglebr} shows that for arbitrary knots $A$ and $B$ we have $$\dt_{\delta}(A, B)\ge\frac{|\e_2(A)-\e_2(B)|}{n-1}.$$ 
Note that if a knot $W$ is a connected sum of knots $A$ and $B$, then it follows from the construction of a $2$-fold branched cover that $\Sigma_2(W)$ is a connected sum of~$\Sigma_2(A)$ and~$\Sigma_2(B)$. This implies that
$${\bf H}_1(\Sigma_2(W))={\bf H}_1(\Sigma_2(A))\oplus{\bf H}_1(\Sigma_2(B)).$$ By construction and by the Schubert theorem (see~\cite{Schu49}, \cite[p.~150]{Ro03}), 
$Q^{i}$ is a connected sum of $i$ copies of $Q$. Then we have
$${\bf H}_1(\Sigma_2(Q^{i}))=\bigoplus_i{\bf H}_1(\Sigma_2(Q)),$$
and therefore $\e_2(Q^i)\geq i$. Note that since each vertex $V_i'$ of the path $\zeta'$ is a connected sum of $V_i$ and $Q^{(n-1)r+e}$, we have
$${\bf H}_1(\Sigma_2(V_i'))={\bf H}_1(\Sigma_2(V_i))\oplus{\bf H}_1(\Sigma_2(Q^{(n-1)r+e}))$$
and therefore $$\dt_{\delta}(V_i',E)\ge\frac{|\e_2(V_i')-e|}{n-1}\ge\frac{|(n-1)r+e-e|}{n-1}\ge r.$$ Moreover, since each vertex $W_i$ of the path $\gamma'$ is a connected sum of $K$ and $Q^i$, we have
$${\bf H}_1(\Sigma_2(W_i))={\bf H}_1(\Sigma_2(K))\oplus{\bf H}_1(\Sigma_2(Q^i)).$$
If $i\ge (n-1)r+e$, then $$\dt_{\delta}(W_i, E)\ge\frac{|\e_2(W_i)-e|}{n-1}\ge\frac{|i-e|}{n-1}\ge\frac{|(n-1)r+e-e|}{n-1}\ge r.$$ Suppose $i<(n-1)r+e$ and $\dt_{\delta}(W_i,E)<r$, then by cons\-truc\-tion of $\gamma'$ and by definition of distance we have $\dt_{\delta}(K,W_i)\le i$ and by the triangle inequality we have
$$nr+e\le \dt_{\delta}(K,E)\le\dt_{\delta}(K,W_i)+\dt_{\delta}(W_i, E)<i+r<(n-1)r+e+r=nr+e,$$
which leads to a contradiction. Therefore if $i<(n-1)r+e$ then it must be~$\dt_{\delta}(W_i, E)\ge r$. Thus, we have obtained that any vertex of the path $\gamma'\cup\,\zeta'$, connecting $K$ and $V_0'$ (where $V_0'$ is a connected sum of $E$ and $Q^{(n-1)r+e}$, and therefore does not depend on the choice of $K$), lies outside the ball $\B_{r-1}^{\delta}(E)$. Note that $K$ is chosen arbitrarily among those such that \mbox{$nr+e\leq\dt_{\delta}(K,E)<\infty$}. This remark completes the proof.
\end{proof}



\begin{proof}[Proof of Theorem 1]

Let us show that for any $\sfrac{p}{q}\in\big(\mathbb{Q}\setminus\{\sfrac{0}{1}\}\big)\cup\{\sfrac{1}{0}\}$ there is a knot $K_{\sfrac{p}{q}}\in \Sf_1^{\ra{p}{q}}(\U)$ such that $\e_2(K_{\sfrac{p}{q}})=1.$ By \nameref{basiclem}, this implies that 
the number of $\BU$-ends of each connected component of $\G(\delta,\bK)$ is equal to one for any rational move $\delta$.

\begin{figure}[H]
\center{\includegraphics[width=0.57\textwidth]{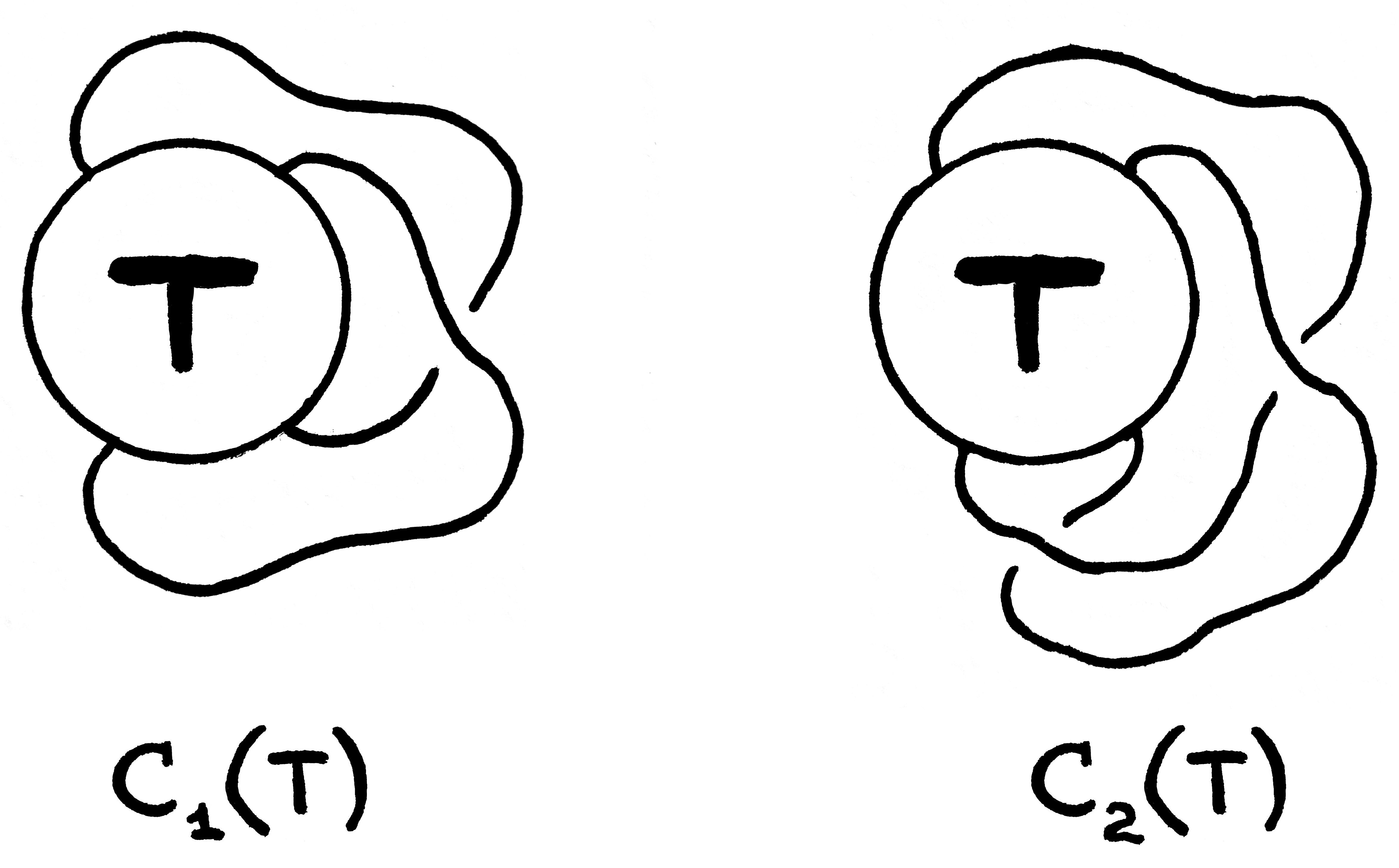}}
\caption{\ti{Closures}} 
\label{Fig19}
\end{figure}

\medskip
First, we introduce two special closures of rational tangles. Figure~\ref{Fig19} shows two closures of a rational tangle~$T$. We denote the results of these closures by~$\C_1(T)$ and~$\C_2(T)$, respectively. It is easy to see that if $\F(T)=p/q$ then $\C_1(T)=\N(S)$ and $\C_2(T)=\N(R),$
where $S$ and $R$ are rational tangles such that $$\F(S)=\frac{p+q}{q},\,\,\,\,\,\,\,\,\,\,\,  \F(R)=\frac{p+2q}{p+q}. $$

\medskip
Let $\sfrac{p}{q}\in\mathbb{Q}\cup\{\sfrac{1}{0}\}$, where $p$ and $q$ are relatively prime, and let $T$ be a rational tangle defining a local-move-pattern of $\ra{p}{q}$-move. Let us find the corresponding knot $K_{\sfrac{p}{q}}$ for all possible cases.

\medskip
\noindent\fbox{Let $pq>0$.} If $p$ is odd and $q$ is odd then we can assume that $K_{\sfrac{p}{q}}=\C_2(T)$ because
        \begin{itemize}
            \item $\C_2(T)=\N(R)$, where $\F(R)=\sfrac{(p+2q)}{(p+q)} $, and in this case $p+2q$ is odd, that is, according to Lemma~\ref{lem:odev}, $K_{\sfrac{p}{q}}$ is a knot,
            \item it is easy to see that $K_{\sfrac{p}{q}}\in \Sf_1^{\ra{p}{q}}(\U)$,
            \item by Lemma~\ref{lem:dbc} we have ${\bf H}_1(\Sigma_2(\N(R)))=\mathbb{Z}_{(p+2q)}$ and therefore $\e_2(K_{\sfrac{p}{q}})=1$, since  $p+2q>1$.
        \end{itemize}

\noindent if $p$ is odd and $q$ is even, or vice versa, then we can assume that $K_{\sfrac{p}{q}}=\C_1(T)$ because
        \begin{itemize}
            \item $\C_1(T)=\N(S)$, where $\F(S)=\sfrac{(p+q)}{q} $, and in this case $p+q$ is odd, that is, according to Lemma~\ref{lem:odev}, $K_{\sfrac{p}{q}}$ is a knot,
            \item it is easy to see that $K_{\sfrac{p}{q}}\in \Sf_1^{\ra{p}{q}}(\U)$,
            \item by Lemma~\ref{lem:dbc} we have ${\bf H}_1(\Sigma_2(\N(S)))=\mathbb{Z}_{(p+q)}$ and therefore $\e_2(K_{\sfrac{p}{q}})=1$, since  $p+q>1$.
        \end{itemize}

\medskip 
\noindent\fbox{Let $pq<0$.} It follows from the definition that $-\F(T)=\F(-T)$, that is, the tangle with fraction $\sfrac{a}{b}$ is a mirror image of the tangle with fraction $\sfrac{-a}{b}$. For clarity, we introduce the notation $\sfrac{p}{q}=\sfrac{-p'}{q'}$, where $p',q'>0$. We can repeat all the previous reasoning for this case by changing all crossings in the constructions of closures $\C_1$ and $\C_2$. Further, fractions $\sfrac{-(p'+2q')}{(p'+q')}$ and $\sfrac{-(p'+q')}{q'}$ appear in a similar way (only the sign changes compared to the previous case).
The numerators of these fractions are also always odd and not equal to $-1$. The corresponding required knot also lies in $ \Sf_1^{\ra{-p'}{q'}}(\U)$ for the same reasons (crossings change does not affect them). And since $\Ll(-a,b)=\Ll(a,-b)$, the proper\-ty~$\e_2(K_{\sfrac{-p'}{q'}})=1$ is also preserved.

\medskip
\noindent\fbox{Let $pq=0$.} In this case, $T$ is $[\infty]$, and $\ra{p}{q}$-move is $\Hm(2)$-move. It is easy to see that in this case we can assume that $K_{\sfrac{p}{q}}$ is the trefoil, denoted by $3_1$. It is well known that ${\bf H}_1(\Sigma_2(3_1))=\mathbb{Z}_{3}$ (see~\cite[p.~304]{Ro03}). It is easy to see that~$3_1\in\Sf_1^{\ra{1}{0}}(\U)$.
This argument completes the proof.
\end{proof}

\begin{proof}[Proof of Theorem~\ref{thm2}]

Let us show that for any $n\in\mathbb{N}$ there is a knot $K_n$ in~$\Sf_1^{\Cm(n)}(\U)$ such that $|\Delta_{K_n}(-1)|>1$.
Lemma~\ref{card} shows that in this case we have $\e_2(K_{n})\ge1$. By \nameref{basiclem}, this implies that 
the number of \mbox{$\BU$-ends} of each connected component of $\G(\Cm(n),\bK)$ is equal to one for any $n\in \mathbb{N}$. 

\begin{figure}[H]
\center{\includegraphics[width=0.75\textwidth]{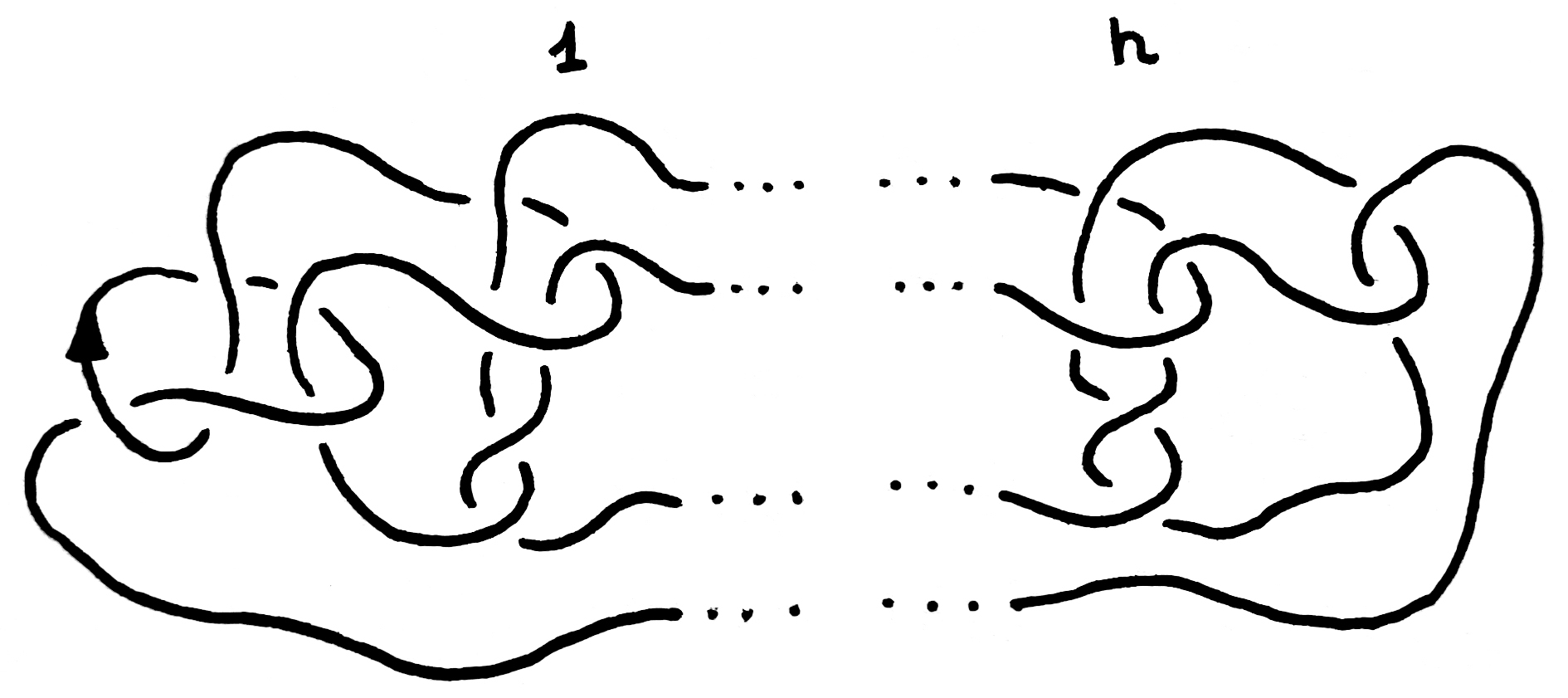}}
\caption{\ti{A family of oriented knots}} 
\label{Fig20}
\end{figure}

\medskip
Let $K_n$ be such a family of oriented knots in $S^3$ as shown in Figure~\ref{Fig20} for~$n\in\mathbb{N}$. Let us calculate the Conway polynomial of $K_n$. For $n=1$ it is easy to see that~$\nabla(K_1)(z)=-z^6-z^4+1$. For any $n>1$ we have
$$\nabla(K_n)(z)=1-z\nabla(K_n^1)(z), \,\,\, \nabla(K_n^1)(z)=\nabla(K_n^2)(z)-z,$$
$$\nabla(K_n^2)(z)=0+z\nabla(K_n^3)(z),$$
where $K_n^1$, $K_n^2$, and $K_n^3$ are the links shown in Figure~\ref{Fig21} (we only show the part of each of these links that changes during the calculation).

\begin{figure}[H]
\center{\includegraphics[width=0.8\textwidth]{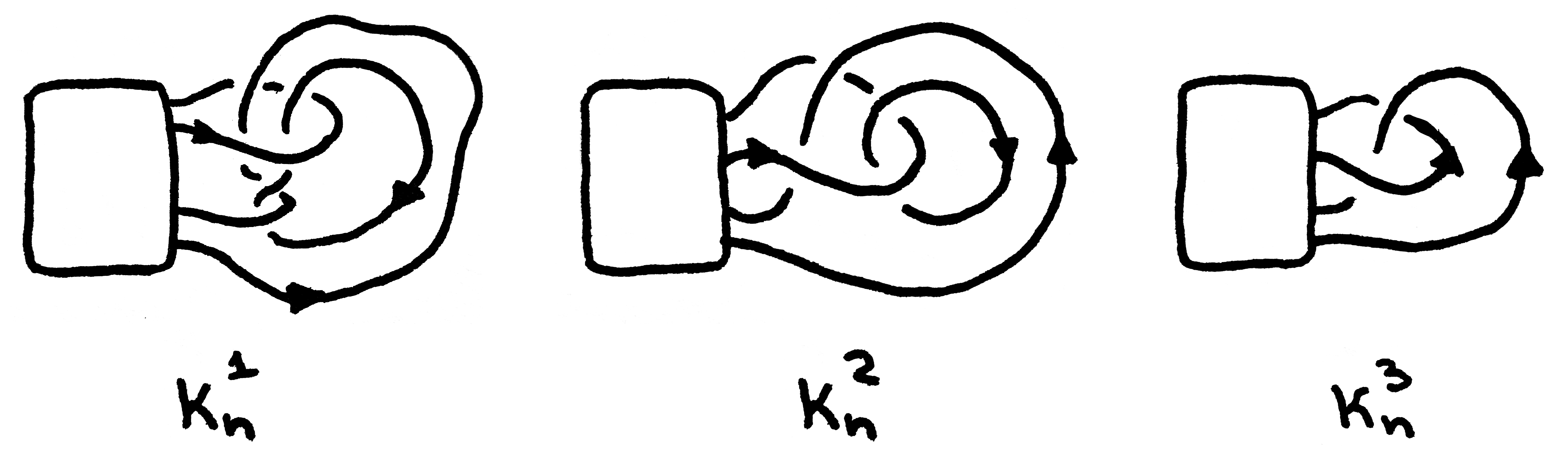}}
\caption{\ti{Intermediate calculations}} 
\label{Fig21}
\end{figure}

Note that $K_n^3$ is ambient isotopic to $K_{n-1}$. Then we have
$$\nabla(K_n)(z)=1-z\nabla(K_n^1)(z)=1-z(\nabla(K_n^2)(z)-z)=$$
$$=1-z(z\nabla(K_{n-1})(z)-z)=1+z^2-z^2\nabla(K_{n-1})(z).$$
By Lemma~\ref{changev} and by forgetting the orientation, we obtain
$$\Delta_{K_n}(-1)=\Delta_{K_n}(i^2)=\nabla(K_n)(2i)=-3+4\nabla(K_{n-1})(2i).$$
Note that $|\Delta_{K_1}(-1)|=49$. Hence, by induction we have $|\Delta_{K_n}(-1)|>1$ for any $n\in \mathbb{N}$. It is easy to see that \mbox{$K_n\in\Sf_1^{\Cm(n)}(\U)$} for any $n\in\mathbb{N}$. This argument completes the proof.
\end{proof}


\begin{proof}[Proof of Theorem~\ref{thm3}]
It is easy to see that $e_2(3_1)=1$, see~\cite[p.~304]{Ro03}. Figure~\ref{Fig22} shows that $3_1\in\Sf_1^{\Hm(n)}(\U)$ for any $n\in\mathbb{N}$. By \nameref{basiclem}, this implies that 
the number of $\BU$-ends of each connected component of $\G(\Hm(n),\bK)$ is equal to one for any $n\in \mathbb{N}$. This argument completes the proof.
\end{proof}

\begin{figure}[H]
\center{\includegraphics[width=0.65\textwidth]{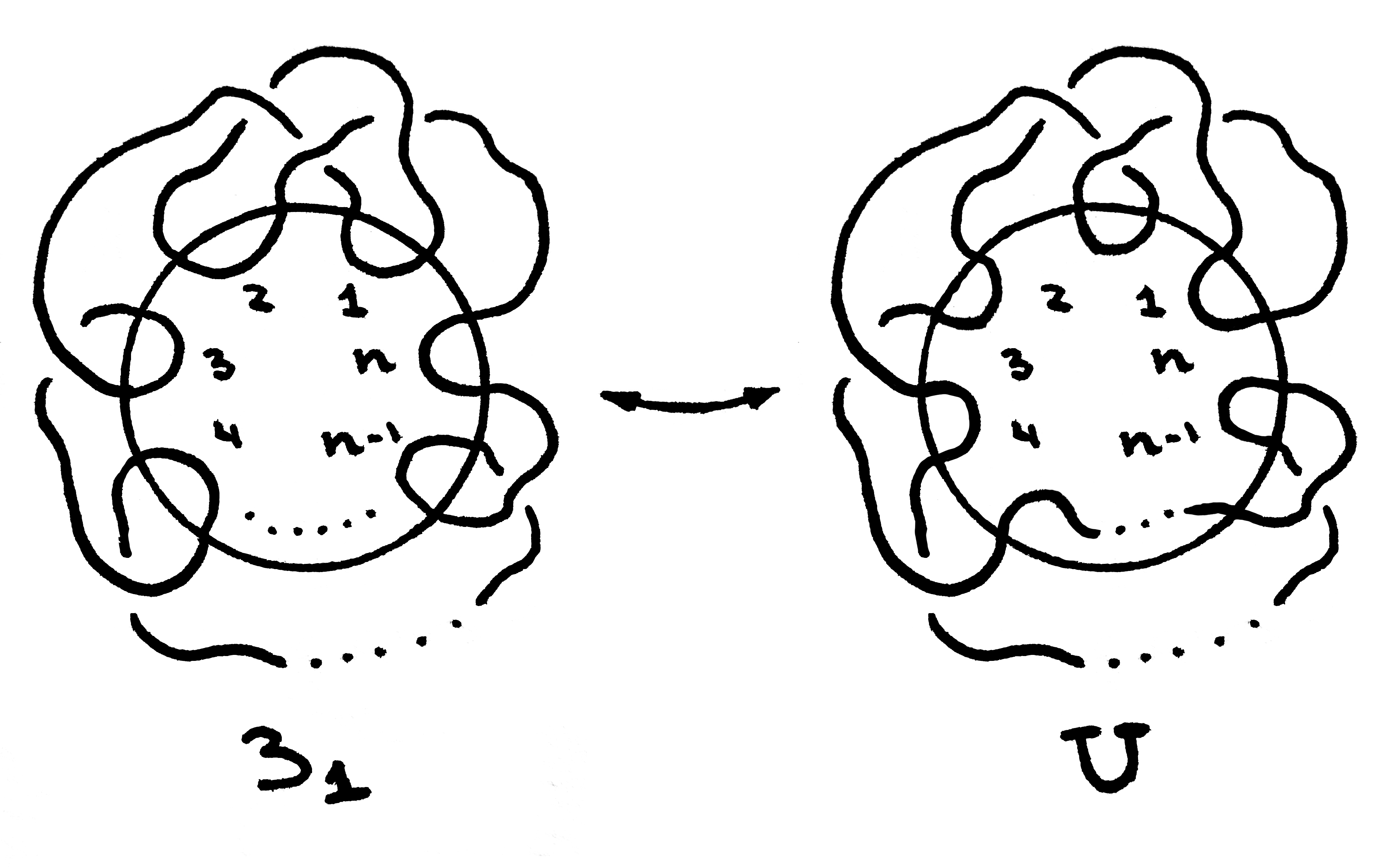}}
\caption{\ti{$3_1\in\Sf_1^{\Hm(n)}(\U)$}} 
\label{Fig22}
\end{figure}


\begin{proof}[Proof of Theorem~\ref{thm4}]
Let $X$ be a finite subset of $\Vrt(\G(\delta,\bK))$, and let $C$ be a connected component of $\G(\delta,\bK)$. Let us show that \mbox{$C\sm (X\cap\Vrt(C))$} is a connected graph. This implies that 
the number of $\FI$-ends and the number of~$\FU$-ends of each connected component of $\G(\delta, \bK)$ are equal to one.

\medskip
Let $K,S\in \Vrt(C)\sm (X\cap\Vrt(C))$, and let $\gamma$ be a path in $C$ connecting $K$ and $S$. Note that $$A=\{Q\in\Vrt(\G(\delta,\bK))\,\,|\,\,Q\in\Sf_1^{\delta}(\U),\,Q\notin X\}$$ is an infinite set. Therefore there is a vertex $Q\in A$ such that \mbox{$\Vrt(\gamma(Q))\cap X=\varnothing$,} where $\gamma(Q)$ is the path obtained by "shifting" $\gamma$ by $Q$. Let $e$ be an edge incident to vertices $\U$ and $Q$, and let $e(K)$ and $e(S)$  be edges such that $e(K)$ is the edge obtained by "shifting" $e$ by $K$, $e(S)$ is the edge obtained by "shifting" $e$ by $S$, and $e(K)\cup\gamma(Q)\cup e(S)$ is a path connecting $K$ and $S$ (such "shifts" can always be obtained by choosing an appropriate gluing homeomorphism). Since no vertex of~$e(K)\cup\gamma(Q)\cup e(S)$ lies in $X$, we can assume that it connects $K$ and $S$ as vertices of $C\sm (X\cap\Vrt(C))$. This argument completes the proof.
\end{proof}
\begin{remark}
Note that $\Sf_1^{\Cm(n)}(\U)$ and $\Sf_1^{\Hm(n)}(\U)$ are infinite sets for any $n\in\mathbb{N}$, see~\cite{O06} and~\cite{ZYL17} .
\end{remark}


\end{document}